\documentclass[tbtags,reqno]{amsart}
\usepackage{geometry}
\usepackage{graphicx, subfigure}
\usepackage{epstopdf}
\usepackage{caption}

\usepackage{xfrac}
\usepackage{tikz-cd}
\usetikzlibrary{decorations.pathmorphing}
\usepackage{mathtools}
\usepackage{cleveref}
\usepackage{mathrsfs}
\usepackage{latexsym}
\usepackage{amsthm,amsopn,tabmac,amsfonts,amssymb,epsfig,color}
\numberwithin{equation}{section}
\makeatletter
\renewcommand{\subsubsection}{\@startsection
{subsubsection}
{3}
{0mm}
{\baselineskip}
{-0.5\baselineskip}
{\normalfont\normalsize\bfseries}}
\makeatother

\newtheorem{theorem}{Theorem}
\newtheorem{lemma}[theorem]{Lemma}
\newtheorem{proposition}[theorem]{Proposition}
\newtheorem{example}[theorem]{Example}

\newtheorem{corollary}[theorem]{Corollary}
\newtheorem{definition}[theorem]{Definition}

\newtheorem{remark}[theorem]{Remark}

\title[Cauchy identities for skew Ferrers shapes via RSK and keys]
{Cauchy identities for skew Ferrers shapes via RSK and keys}

\begin{document}

\author{Luis Pena}
\address{Instituto de Matem\'aticas, Universidad de
Talca, 2 norte 685, Talca, Chile.}
\email{luis.cardenas@utalca.cl }

\begin{abstract}
Let $\mu\subseteq\lambda\subseteq(m^n)$. We characterize the image under the ordinary Robinson--Schensted--Knuth correspondence of matrices supported on the skew Ferrers diagram $\lambda/\mu$.
The outer boundary determines an upper bound on the right key of the insertion tableau, while the inner boundary determines a lower bound on its left key; both bounds depend on the keys of the recording tableau.
This yields tableau expansions of skew Ferrers Cauchy kernels using the standard basis polynomials of Lascoux and Schützenberger, indexed by intervals in Bruhat order.

The proof first treats ordinary Ferrers diagrams.
Using the supremum characterization of right keys from earlier work, we follow the $\lambda$-dependent bounds through single RSK insertions.
When $\lambda$ has repeated parts, these weak column bounds need not form a semistandard tableau.
Strictification determines a set $\operatorname{Comp}(\lambda)$ of admissible weak compositions and, for each $\alpha\in\operatorname{Comp}(\lambda)$, a composition $\alpha^\lambda$.
Ordinary RSK then gives a weight-preserving bijective realization of the expansion
\[
\prod_{(i,j)\in\lambda}\frac{1}{1-x_i y_j}
=
\sum_{\alpha\in\operatorname{Comp}(\lambda)}
\hat K_\alpha(x)K_{\alpha^\lambda}(y),
\]
where $\hat K_\alpha$ and $K_\alpha$ denote Demazure atoms and key polynomials, respectively.

We also give a direct admissibility criterion and a parking procedure for computing $\alpha^\lambda$.
After translating conventions, these agree with the admissibility condition and half-bubble-sort construction of Feigin, Khoroshkin, and Makedonskyi.
The staircase and truncated-staircase identities follow as special cases.
Finally, we extend the weak-bound construction to an infinite alphabet, where strictification need not exist, and derive the infinite-variable Cauchy identity for the $m$-symmetric Schur functions.
\end{abstract}


\maketitle

\section{Introduction}

The Robinson--Schensted--Knuth correspondence is one of the central bijections in algebraic combinatorics.
It associates with every lexicographically ordered biword a pair $(P,Q)$ of semistandard Young tableaux of the same shape, while preserving the weights of the two
rows of the biword.
At the level of generating functions, this gives the classical Cauchy identity \cite{Fulton1996,Stanley_Fomin_1999}
$$
\prod_{i,j}\frac{1}{1-x_i y_j}
=
\sum_{\nu}s_\nu(x)s_\nu(y).
$$

The common shape of the two tableaux is what indexes the two Schur factors by the same partition.

A natural refinement is obtained by restricting the allowed biletters.
If $\lambda$ is a partition, regarded as a Ferrers diagram, one may restrict to biletters $\binom{i}{j}$ satisfying $(i,j)\in\lambda$.
More generally, for $\mu\subseteq\lambda\subseteq(m^n)$ one may restrict to the skew Ferrers diagram $\lambda/\mu$.
The corresponding generating function is
$$
\prod_{\substack{1\leq i\leq n\\
                  \mu_i<j\leq\lambda_i}}
\frac{1}{1-x_i y_j}.
$$
The main question of this article is to determine explicitly which pairs of tableaux arise under ordinary RSK from matrices supported on such a skew Ferrers diagram, and to use this description to obtain nonsymmetric Cauchy expansions of these kernels.

For ordinary Ferrers diagrams, this problem is closely related to Demazure characters and Demazure atoms.
If $\alpha$ is a weak composition, the tableaux satisfying $K_+(T)=K(\alpha)$ generate the Demazure atom $\widehat K_\alpha$, whereas those satisfying $K_+(T)\leq K(\alpha)$ generate the key polynomial $K_\alpha$ \cite{Lascoux1990Schutzenberger}.
Thus restricting the support of a biword from a rectangle to a Ferrers diagram refines the common-shape condition in the classical Cauchy identity by inequalities between keys.

This phenomenon has appeared in several forms.
For staircase shapes, Lascoux obtained a nonsymmetric Cauchy identity in Demazure characters and Demazure atoms \cite{Lascoux2003}; see also \cite{FuLascoux2009} for a treatment in terms of divided-difference operators and extensions to the other classical types.
Mason introduced an analogue of RSK based on semi-skyline augmented fillings \cite{Mason2008}.
Azenhas and Emami used this correspondence for truncated staircases \cite{AzenhasEmami2015} and later treated northwest and southeast near-staircase shapes using Mason's correspondence and growth diagrams \cite{AzenhasEmami2015Growth,AzenhasEmami2014NWSE}.
A crystal-theoretic interpretation and bijective proof of the staircase identity were given by Choi and Kwon \cite{ChoiKwon2018}, and related restrictions of RSK were studied from the viewpoint of crystals and last-passage percolation by Azenhas, Gobet, and Lecouvey \cite{AzenhasGobetLecouvey2024}.

For arbitrary Ferrers diagrams, Lascoux obtained a divided-difference-operator expansion of the corresponding kernel \cite[Theorem~7]{Lascoux2003}.
More recently, Feigin, Khoroshkin, and Makedonskyi obtained the corresponding right Cauchy expansion through standard filtrations and representation-theoretic methods \cite{FeiginKhoroshkinMakedonskyi2026}.
Their indexing is expressed in terms of admissible compositions and a half-bubble-sort construction.
Khoroshkin and Makedonskyi subsequently gave an independent approach based on bubble-sort combinatorics and generalized Howe duality \cite{KhoroshkinMakedonskyi2025}.
These approaches give the Cauchy expansion but do not formulate the restriction as an explicit characterization of the pairs of semistandard tableaux arising under ordinary RSK.
For ordinary Ferrers shapes, after translating conventions, the indexing obtained in \cite{FeiginKhoroshkinMakedonskyi2026} agrees with ours; the main difference lies in the combinatorial model and in the route used to obtain the expansion.
While \cite{FeiginKhoroshkinMakedonskyi2026} uses serpentines and a Pieri rule based on \cite{AssafQuijada2018} which use Kohert diagrams, we use the classical tableaux model of key polynomials and demazure atoms which involves right keys. This model is quite flexible and allows the extension to skew shapes via left keys using only basic machinery.

Our first ingredient is such a characterization for an arbitrary Ferrers diagram $\lambda$.
We define a pair $(P,Q)$ to be $\lambda$-admissible by a family of $\lambda$-dependent inequalities between right keys.
The proof uses the description, developed in \cite{mSchurt0}, of each column of the right key as the supremum of
the decreasing subwords of a suitable column subtableau.
This description is compatible with a single RSK insertion and allows the admissibility inequalities to be propagated locally.

The bounds associated with $\lambda$ are naturally weak column bounds.
When $\lambda$ has repeated parts they need not themselves form a semistandard tableau.
We therefore introduce a strictification procedure, replacing each weak column by the largest strictly increasing column below it.
This produces a set $\operatorname{Comp}(\lambda)$ of admissible weak compositions and, for every $\alpha\in\operatorname{Comp}(\lambda)$, a composition
$\alpha^\lambda$.
Ordinary RSK then restricts to a weight-preserving bijection
$$
\operatorname{RSK}:\mathcal B_\lambda
\longrightarrow
\bigcup_{\alpha\in\operatorname{Comp}(\lambda)}
\mathcal S^\lambda(\alpha)\times\mathcal S(\alpha),
$$
and taking generating functions gives
$$
\prod_{(i,j)\in\lambda}\frac{1}{1-x_i y_j}
=
\sum_{\alpha\in\operatorname{Comp}(\lambda)}
\widehat K_\alpha(x)K_{\alpha^\lambda}(y).
$$
Thus the arbitrary-Ferrers Cauchy expansion is realized directly by ordinary RSK.

The skew case introduces a second boundary and, correspondingly, a second key condition.
Let
$$
\mu^\vee
=
(m-\mu_n,m-\mu_{n-1},\ldots,m-\mu_1)
$$
be the Ferrers diagram obtained by rotating the complement of $\mu$ in $(m^n)$ through $180^\circ$.
If $A^\circ$ denotes the $180^\circ$ rotation of an $n\times m$ matrix $A$, then
$$
\operatorname{supp}(A)\subseteq\lambda/\mu
\quad\Longleftrightarrow\quad
\operatorname{supp}(A)\subseteq\lambda
\quad\text{and}\quad
\operatorname{supp}(A^\circ)\subseteq\mu^\vee.
$$
Moreover, if $\operatorname{RSK}(A)=(P,Q)$, then the standard compatibility of RSK with Schützenberger evacuation gives
$$
\operatorname{RSK}(A^\circ)
=
\bigl(\operatorname{ev}_m(P),\operatorname{ev}_n(Q)\bigr).
$$
Consequently, our Ferrers characterization applied simultaneously to $A$ and $A^\circ$ yields the main skew result: $\operatorname{supp}(A)\subseteq\lambda/\mu$ if and only if $(P,Q)$ is $\lambda$-admissible and $\bigl(\operatorname{ev}_m(P),\operatorname{ev}_n(Q)\bigr)$
is $\mu^\vee$-admissible.

The two conditions have a natural interpretation in terms of the two keys of the insertion tableau.
The outer boundary $\lambda$ determines an upper bound on $K_+(P)$, while the inner boundary $\mu$, after rotation and evacuation, determines a lower bound on $K_-(P)$.
Hence the insertion tableaux occurring for a fixed pair of extremal keys of $Q$ satisfy inequalities of the form
$$
K(\alpha_\mu)
\leq K_-(P)
\leq K_+(P)
\leq K(\beta^\lambda).
$$
Their generating functions are therefore the two-sided standard bases of Lascoux and Schützenberger, indexed by intervals in Bruhat
order.
This gives positive tableau expansions of the skew Ferrers kernels and explains why passing from a Ferrers diagram to a skew Ferrers diagram naturally replaces a one-sided Demazure condition by simultaneous lower and upper key bounds.

The transformations associated with the two boundaries are governed by the same construction.
We give a direct criterion for $\alpha\in\operatorname{Comp}(\lambda)$ and a parking procedure for computing $\alpha^\lambda$.
After translating conventions, this criterion and construction agree with the admissibility condition and half-bubble-sort procedure of Feigin, Khoroshkin, and Makedonskyi.
The lower-bound index associated with $\mu$ is obtained from the same construction by rotating the diagram through $180^\circ$, so no second algorithm is required.

Several known identities are recovered as special cases.
The staircase and truncated-staircase kernels specialize to the corresponding nonsymmetric Cauchy identities above.
A further specialization recovers the Cauchy identity for the $m$-symmetric Schur functions at $t=0$ obtained in \cite{mSchurt0}.
The weak-bound formulation also extends naturally to an infinite alphabet, even when the bounds cannot be strictified into key tableaux, and yields the corresponding infinite-variable $m$-symmetric Cauchy identity.

The paper is organized as follows.
Section~\ref{SecPrelim} recalls the necessary background on tableaux, key polynomials, Demazure atoms, RSK, and left and right keys.
Section~\ref{SecKeylambda} develops the $\lambda$-dependent weak bounds and their strictification, proves the restricted RSK correspondence, and derives the ordinary Ferrers Cauchy identity.
Section~\ref{SecDirectAlphaLambda} gives the existence criterion and parking construction for $\alpha^\lambda$.
Section~\ref{SecSkewFerrers} extends the characterization to skew Ferrers diagrams and derives the corresponding Cauchy expansions.
Section~\ref{SecSpecializations} compares our construction with half-bubble sort and Lascoux's divided-difference-operator expansion, and treats the truncated-staircase and $m$-symmetric specializations.
Finally, Appendix~\ref{AppInfiniteAlphabet} extends the weak-bound construction to infinite alphabets and derives the infinite-variable $m$-symmetric Cauchy identity.

\section{Preliminaries}
\label{SecPrelim}

This section gathers the definitions and notation used throughout the article. Throughout,
\[
\mathbb N=\{0,1,2,\ldots\},
\qquad
\mathbb{Z}_{>0}=\{1,2,3,\ldots\},
\]
and, for \(n\in\mathbb{Z}_{>0}\), we write
\[
[n]=\{1,\ldots,n\}.
\]

\subsection{Partitions, key polynomials, and Demazure atoms.} \label{ssecsym}

A partition $\lambda=(\lambda_1 \geq \lambda_2 \geq \cdots \geq \lambda_k>0)$ is a weakly decreasing sequence of positive integers.
Its degree is $|\lambda|=\lambda_1+\cdots+\lambda_k$ and its length is $\ell(\lambda)=k$. We represent $\lambda$ by a Young diagram with $\lambda_i$ lattice squares in the $i$-th row, from top to bottom (English notation).
Any lattice square $(i,j)$ in the $i$-th row and $j$-th column of a Young diagram is called a cell.
The conjugate partition $\lambda'$ is the partition whose diagram is obtained from that of $\lambda$ by interchanging rows and columns.
A composition $\pmb a=(a_1,\ldots,a_r)$ is a finite sequence of nonnegative integers.

 Let the exchange operator $K_{i,j}$ be such that
$$K_{i,j} f(\dots, x_i,\dots,x_j,\dots)= f(\dots, x_j,\dots,x_i,\dots).$$

We will need two families of divided-difference-operators $\pi_{\sigma},\hat \pi_\sigma$, where
$$
\hat \pi_i = \frac{x_{i+1}}{x_i-x_{i+1}} (1-K_{i,i+1}).
$$
We also have that $\pi_i=\hat \pi_i+1$.
One checks directly that $\hat \pi_i^2=-\hat \pi_i$ and $\pi_i^2=\pi_i$.
We define $\pi_{\mathrm{Id}}=\hat{\pi}_{\mathrm{Id}}=\mathrm{Id}$ and, recursively, if $\sigma'=s_i\sigma$ and $\sigma^{-1}(i)<\sigma^{-1}(i+1)$ (which is to say that $\ell(\sigma')=\ell(\sigma)+1$), then
\[
\pi_{\sigma'}=\pi_i\pi_\sigma,
\qquad
\hat{\pi}_{\sigma'}=\hat{\pi}_i\hat{\pi}_\sigma.
\]

Using these operators, we can define the key polynomials $K_{\alpha}(x_1,\dots,x_n)$ and dual key polynomials ${\hat K_{\alpha}(x_1,\dots,x_n)}$ recursively as:
\begin{equation} \label{defdualkey}
\hat K_{\alpha}(x)= \left \{ 
\begin{array}{ll}
x^{\alpha} & {\rm if~} \alpha_1\geq\alpha_2\geq\cdots\geq\alpha_n \\
\hat \pi_i \hat K_{s_i \alpha}(x) & {\rm if~} \alpha_i < \alpha_{i+1}
\end{array} \right .,
\end{equation}
and
\begin{equation}
\label{defkey}
 K_{\alpha}(x) = \left \{ 
\begin{array}{ll}
x^{\alpha} & {\rm if~} \alpha_1\geq\alpha_2\geq\cdots\geq\alpha_n \\
 \pi_i  K_{s_i \alpha}(x) & {\rm if~} \alpha_i < \alpha_{i+1}
\end{array} \right .,
\end{equation}
where  $\alpha$ is the composition $\alpha=(\alpha_1,\dots,\alpha_n)$.  Key polynomials were introduced in \cite{Demazure1974a,Demazure1974b}, and are also known as Demazure characters, while the dual key polynomials are also known as Demazure atoms.

\subsection{Tableaux, words and the RSK correspondence} \label{ssectab}

For any totally ordered alphabet $\mathcal B$, let $\operatorname{Tab}_{\mathcal B}$ be the set of semistandard Young tableaux in the alphabet $\mathcal B$, that is, the set of fillings of a Young diagram with letters of $\mathcal B$ that are weakly increasing along rows (from left to right) and strictly increasing along columns (from top to bottom).
The elements of $\operatorname{Tab}_{\mathcal B}$ will simply be called semistandard tableaux, and we will denote by $\operatorname{Tab}_{\mathcal B}(\lambda)$ the set of those of shape $\lambda$.
Skew tableaux, that is, fillings of a skew diagram $\lambda/\mu$ obeying the same conditions, will always be qualified as such.
We let $T(i,j)$ stand for the entry of $T$ in cell $(i,j)$, and we let $x^T$ be the monomial in which the power of $x_i$ is the number of occurrences of the letter $i$ in $T$.
Two tableaux $P$ and $Q$ of the same shape $\lambda$ are such that $P \leq Q$ if and only if $P(i,j) \leq Q(i,j)$ for any $(i,j) \in \lambda$.
Given a tableau $T$, we let $w(T)$ be the word obtained by reading the entries of $T$ from left to right and from bottom to top.
The weight of a tableau $T$, denoted by $\operatorname{wt}(T)$, is the sequence of multiplicities of its entries; that is, the component indexed by $b\in\mathcal B$ is the number of occurrences of the letter $b$ in $T$.
In particular, if $\mathcal B=\{1,\ldots,n\}$ and $\operatorname{wt}(T)=(a_1,\ldots,a_n)$, then

\[
x^T=x^{\operatorname{wt}(T)}=x_1^{a_1}\cdots x_n^{a_n}.
\]

For $x=(x_1,\ldots,x_n)$, the Schur polynomial is
\[
s_\lambda(x_1,\ldots,x_n)
=
\sum_{T\in\operatorname{Tab}_{[n]}(\lambda)}
x^{\operatorname{wt}(T)}.
\]
The corresponding Schur function is obtained by allowing entries in the infinite alphabet $\mathbb{Z}_{>0}$.

Recall that two words $u$ and $w$ are Knuth equivalent, which we denote by $u \equiv w$, if one can be obtained from the other by a sequence of elementary Knuth transformations $(K')$ or $(K'')$:
\begin{align*}
  (K'): \quad y x z& \equiv  y z x, {\rm~for~}x < y \leq z \\
  (K''): \quad x zy & \equiv  z x y, {\rm~for~}x \leq y < z. 
\end{align*}

We finally recall the row insertion algorithm and the RSK correspondence, referring to \cite{Stanley_Fomin_1999,Fulton1996} for the details.
Given a tableau $T$ and a letter $j$, the tableau $T \leftarrow j$ is obtained by inserting $j$ in the first row of $T$, where it bumps the leftmost entry strictly larger than $j$ (if there is no such entry, $j$ is appended at the end of the row and the algorithm stops); the bumped entry is then inserted in the next row in the same fashion, and so on until an entry is appended at the end of a row.
The entries $c_{k-1},\dots,c_1,c_0(=j)$ successively inserted in this process, where $c_{k-1}$ is the entry appended in row $k$, form what we call the insertion path of $T \leftarrow j$.
Note that $w(T \leftarrow j)$ is Knuth equivalent to $w(T)j$.

Given a biword
$$
\left(
\begin{array}{cccc}
i_1 & i_2 & \cdots & i_r \\
j_1 & j_2 & \cdots & j_r
\end{array}  
\right)
$$
in lexicographic order (that is, such that $i_1 \leq i_2 \leq \cdots \leq i_r$, and such that $j_k\leq j_{k+1}$ whenever $i_k = i_{k+1}$), the RSK correspondence builds a pair $(P,Q)$ of tableaux of the same shape by inserting successively the letters $j_1,\dots,j_r$ in $P$, the tableau $Q$ recording in which cell the shape has grown at each step.  It is a bijection between lexicographically ordered biwords and pairs of tableaux of the same shape.

\subsection{Right key tableaux and their characterizations} \label{seckeys}

The right key is the combinatorial invariant governing the tableau expansions of key polynomials and the admissibility conditions introduced in the next section.
We recall below its relationship with key polynomials and the characterization by decreasing subwords established in \cite{mSchurt0}.
For the classical construction, see \cite{Lascoux1990Schutzenberger} and \cite[Appendix~A5]{Fulton1996}; for direct procedures for computing keys, see~\cite{Willis2013,KushwahaRaghavanViswanath2025}.

A combinatorial formula for the key polynomials \eqref{defkey} in terms of key tableaux was obtained in \cite{Lascoux1990Schutzenberger}, and it is this characterization that we will use throughout the article.
Unless stated otherwise, all tableaux appearing in this subsection are taken in the alphabet $\mathbb Z_{>0}$.

\begin{definition}
    Let $T$ be a tableau with $\ell$ columns, and let $C_1,C_2,\ldots,C_\ell$ denote its columns from left to right.
We say that $T$ is a \emph{key tableau} if
\[
C_i \supseteq C_j
\qquad\text{whenever }1\le i<j\le \ell .
\]
\end{definition}

\begin{example} The tableau 
\tableau[scY]{
1 & 2 & 2 \\
2 & 3 \\
3
}
is a key tableau because $\{1,2,3 \} \supseteq \{ 2,3 \} \supseteq \{2\}$.
\end{example}

\begin{example} The tableau
\tableau[scY]{
1 & 3 \\
2
}
is not a key tableau because $\{ 1,2\}$ does not contain $\{ 3 \}$.
\end{example}

Given a weak composition \(\alpha=(\alpha_1,\ldots,\alpha_n)\in\mathbb N^n\), let \(\alpha^+\) be the partition obtained by rearranging the entries of \(\alpha\) in weakly decreasing order.
We denote by \(K(\alpha)\) the unique key tableau of shape \(\alpha^+\) and content \(\alpha\).

For weak compositions $\alpha$ and $\beta$ with $\alpha^+=\beta^+$, the entrywise order on their key tableaux agrees with Bruhat order:
\[
K(\alpha)\leq K(\beta)
\qquad\Longleftrightarrow\qquad
\alpha\leq\beta.
\]
This follows from the tableau criterion for Bruhat order \cite[Theorem~2.6.3]{BjornerBrenti2005}, applied to the corresponding minimal coset representatives.

The remainder of this subsection is adapted from~\cite[Section 4.2]{mSchurt0}. 
Proofs of the quoted results may be found there.

We recall the following characterization of right key tableaux, based on suprema of decreasing words.

We use the ordered alphabet
\[
\mathcal A_m=\{1,2,\ldots,m\},
\qquad
1<2<\cdots<m.
\]
When no confusion can arise, we omit the subscript and write $\mathcal A$ instead of $\mathcal A_m$.

\begin{definition}
\label{DefOrderWords}
A word $w=w_1 \cdots w_t$
is said to be decreasing if $w_1 > w_2 > \cdots > w_t$.
Let $W( \mathcal{A} )$ be the set of \textit{decreasing} words on $\mathcal{A}$.
Given $v,w \in W ( \mathcal{A} )$, with $v= v_{1} \ldots v_{s} , w = w_{1} \ldots w_{t}$, we say that $v \leq w$ iff $s \leq t$ and $v_{i} \leq  w_{i}$ for $1 \leq i \leq s$. We write $v\subseteq w$ to indicate that $v$ is a subword of $w$.
For a nonempty set of decreasing words $A$, define $t = \max \{ |b|, b \in A\}$, and set $\sup A = w_{1} \cdots w_{t}$ where
$$
w_{i} = \max \{ b_{i} | b_{1} \cdots b_{j} \in A, j \geq i\}.
$$
\end{definition}
We denote by $\overline{W}(\mathcal A)$ the set of weakly decreasing words on $\mathcal A$. We extend the order $\leq$ to $\overline{W}(\mathcal A)$ by the same rule: if $v=v_1\cdots v_s$ and $w=w_1\cdots w_t$ are weakly decreasing, then
\[
v\leq w
\quad\Longleftrightarrow\quad
s\leq t
\text{ and }
v_i\leq w_i
\text{ for }1\leq i\leq s.
\]
We use $\sup$ only for sets of strictly decreasing words.
Weakly decreasing words will occur only as bounds.
When we write $v\subseteq w$ for weakly decreasing words, containment is understood in the sense of subwords, with repeated occurrences counted with multiplicity.

\begin{proposition}\cite[Proposition 17]{mSchurt0}.
    Let $A$ be a nonempty set of words in $W( \mathcal{A} )$. Then $\sup A \in W(\mathcal{A})$, that is, $\sup A$ is a decreasing word.
\end{proposition}

\begin{remark}
    The word $\sup A$ defined above is indeed the supremum of $A$ with respect to the partial order on $W(\mathcal{A})$.
\end{remark}

\begin{proposition}\cite[Proposition 19]{mSchurt0}.
  \label{IncImpLeq} If $v$ and $w$ are two decreasing words such that
   $v \subseteq w$, then $v \leq w$.
\end{proposition}

Given a word $w$, we let $W(w)$ be the set of its decreasing subwords, including the empty word $\varepsilon$.
From now on, a tableau $T$ will always be such that either $T \in \operatorname{Tab}_{\mathcal A}$ or $T \in \operatorname{Tab}_{[n]}$.
Recalling that $w(T)$ stands for the reading word of $T$, we let $W(T)$ be the set of decreasing subwords of $w(T)$, that is, $W(T)=W(w(T))$.

\begin{example}
Let 
$$
T = \tableau[scY]{
1 & 3 & 3 & 5 \\
2 & 4 & 6 \\
}
$$
Then $w(T) = 2 4 6 1 3 3 5$, and
$$
W(T) = \{65,63,61,43,41,21,6,5,4,3,2,1,\varepsilon\}.
$$
Therefore $\sup W(T) = 65$.
\end{example}

The whole point of taking a supremum over decreasing subwords is that the result only depends on the Knuth class of the word.
This is the content of the next lemma, in which $(K')$ and $(K'')$ are the elementary Knuth transformations introduced in Section~\ref{ssectab}.
\begin{lemma}\cite[Lemma 22]{mSchurt0}.
  \label{SupKnuthEquivalent}
If $w_1  \equiv w_2$, then
$$
\sup W(w_1) = \sup W(w_2).
$$
\end{lemma}

Since we do not use the classical construction of the right key, we
recall only the following characterization from
\cite[Section~4.2]{mSchurt0}. For a tableau $T$ with $\ell$ columns and $1\leq r\leq\ell$, let $T_r$ be the subtableau of $T$ obtained by considering only the columns of $T$ weakly to the right of column $r$, and set
\[
\mathcal C_r(T)=\sup W(T_r).
\]

\begin{proposition}\cite[Proposition 27]{mSchurt0}.
\label{IgualdadKeys}
The word \(\mathcal C_r(T)\) is the bottom-to-top column word of
column \(r\) of the classical right key \(K_+(T)\). Consequently,
the right key is uniquely determined by $\mathcal C_r(T)$ with $1\leq r\leq\ell$.
\end{proposition}

We conclude this subsection by recalling the tableau formulas for Demazure atoms and key polynomials.

For a weak composition $\alpha$, the Demazure atom indexed by \(\alpha\) is given by
\begin{equation}
\label{EqAtomTableau}
\hat K_\alpha(x_1,\ldots,x_n)
=
\sum_{\substack{
T\in\operatorname{Tab}_{[n]}(\alpha^+)\\
K_+(T)=K(\alpha)}}
x^{\operatorname{wt}(T)},
\end{equation}
whereas the key polynomial indexed by \(\alpha\) is given by
\begin{equation}
\label{EqKeyTableau}
K_\alpha(x_1,\ldots,x_n)
=
\sum_{\substack{
T\in\operatorname{Tab}_{[n]}(\alpha^+)\\
K_+(T)\leq K(\alpha)}}
x^{\operatorname{wt}(T)}.
\end{equation}
These formulas follow from the characterization of Demazure atoms and Demazure characters by right keys; see~\cite{Lascoux1990Schutzenberger}. In particular,
\begin{equation}
\label{EqKeyIntoAtoms}
K_\alpha
=
\sum_{\substack{
\beta^+=\alpha^+\\
K(\beta)\leq K(\alpha)}}
\hat K_\beta.
\end{equation}

\subsection{Left keys and Schützenberger evacuation}
\label{SecLeftKeysEvacuation}

We recall the left-key analogue of the right key and its relation with Schützenberger evacuation.
For a positive integer $r$, let
\[
\operatorname{ev}_r:\operatorname{Tab}_{[r]}
\longrightarrow
\operatorname{Tab}_{[r]}
\]
denote the Schützenberger involution on semistandard tableaux with entries in $[r]$; see, e.g., \cite[Appendix~A.1]{Fulton1996}.
It preserves the shape of a tableau and reverses its weight.
More precisely, if
\[
\omega_r(a_1,\ldots,a_r)=(a_r,\ldots,a_1),
\]
then $\operatorname{wt}\bigl(\operatorname{ev}_r(T)\bigr)
=
\omega_r\operatorname{wt}(T)$.

The left key of a tableau $T\in\operatorname{Tab}_{[r]}$ is denoted by $K_-(T)$, while $K_+(T)$ denotes its right key.
In type $A$, the Lusztig--Schützenberger involution is Schützenberger evacuation, and it exchanges left and right keys; see \cite[Theorem~2.14]{AzenhasGonzalezHuangTorres2024}.
Thus
\[
K_+\bigl(\operatorname{ev}_r(T)\bigr)
=
\operatorname{ev}_r\bigl(K_-(T)\bigr),
\qquad
K_-\bigl(\operatorname{ev}_r(T)\bigr)
=
\operatorname{ev}_r\bigl(K_+(T)\bigr).
\]
Equivalently,
\[
K_-(T)
=
\operatorname{ev}_r\!\left(
K_+\bigl(\operatorname{ev}_r(T)\bigr)
\right).
\]
For $\alpha\in\mathbb N^r$, since evacuation preserves key tableaux and reverses their content, we have
\[
\operatorname{ev}_r\bigl(K(\alpha)\bigr)=K(\omega_r\alpha).
\]
Evacuation reverses the entrywise order on key tableaux of the same shape.
Indeed, if the entries of a column of $K(\alpha)$ are
\[
a_1<\cdots<a_k,
\]
then the corresponding column of $K(\omega_r\alpha)$ has entries
\[
r+1-a_k<\cdots<r+1-a_1.
\]
Consequently,
\[
K(\alpha)\leq K(\beta)
\qquad\Longleftrightarrow\qquad
K(\omega_r\alpha)\geq K(\omega_r\beta).
\]
Since
\[
\operatorname{ev}_r\bigl(K(\alpha)\bigr)=K(\omega_r\alpha),
\]
we obtain, for key tableaux $K$ and $K'$ of the same shape,
\[
K\leq K'
\qquad\Longleftrightarrow\qquad
\operatorname{ev}_r(K')\leq\operatorname{ev}_r(K).
\]
In particular, an upper bound on the right key of an evacuated tableau is equivalent, after applying evacuation, to a lower bound on the left key of the original tableau.

We will also use the standard fact that, for every tableau $T$, $K_-(T)\leq K_+(T)$; see \cite{Lascoux1990Schutzenberger}.

These facts will be used in Section~\ref{SecSkewFerrers} to translate the inner boundary of a skew Ferrers diagram into a lower left-key bound.

\section{$\lambda$-dependent bounds and restricted RSK}
\label{SecKeylambda}

\subsection{$\lambda$-dependent bounds and strictification}

We now introduce the $\lambda$-dependent fillings and the strictification procedure underlying the proof of the main theorem.
To avoid overloading the subscripts, we use the following notation. For a partition $\lambda$, write
\[
\lambda[i]=
\lambda_i,\quad 1\leq i\leq\ell(\lambda) .
\]
When needed, we append trailing zeros to partitions and set $\lambda[i]=0$ for $i>\ell(\lambda)$.
For a boundary contained in $(m^n)$, indexing compositions are regarded as elements of $\mathbb N^n$, and their transforms as elements of $\mathbb N^m$, by appending trailing zeros when necessary.

\begin{definition}
Let $\lambda=(\lambda_1,\ldots,\lambda_n)$ be a partition and let $\alpha\in\mathbb N^n$. Define $B_+^\lambda(\alpha)$ to be the filling of the diagram $\alpha^+$ obtained column by column from $K(\alpha)$.
If column $r$ has length $c_r$, set
\begin{equation}
\label{lambda-key}
\bigl(B_+^\lambda(\alpha)_r\bigr)_i
=
\lambda\left[
\bigl(K(\alpha)_r\bigr)_{c_r+1-i}
\right],
\qquad
1\leq i\leq c_r.
\end{equation}
Thus, the entries of each column of $K(\alpha)$ are read from bottom to top before applying the map $\lambda[\cdot]$.
\end{definition}

\begin{remark}
In general, $B_+^\lambda(\alpha)$ is weakly increasing down its columns under the usual tableau convention, and weakly increasing along its rows.
If $\lambda$ has no repeated parts, then its columns are strictly increasing and $B_+^\lambda(\alpha)$ is an SSYT.
Thus, $B_+^\lambda(\alpha)$ is not generally a semistandard tableau; rather, it is simply a filling of the shape $\alpha^+$.
\end{remark}

\begin{remark}\label{remarklambda}
The construction of $B_+^\lambda(\alpha)$ can be described in terms of an operation on columns.
If the entries in a column of $K(\alpha)$, read from top to bottom, are
\[
w_1,\dots,w_r,
\]
then the corresponding entries in the reversed column of $B_+^\lambda(\alpha)$ are
\[
\lambda[w_r],\dots,\lambda[w_1].
\]
We will refer to the map
\[
w_1\cdots w_r
\longmapsto
\lambda[w_r]\cdots\lambda[w_1]
\]
as the $\lambda$-operation. For a decreasing word
\[
w=w_1\cdots w_r,
\]
the entries of $w^\lambda$ are obtained by applying the boundary function $\lambda[\cdot]$ to the entries of $w$ and then reversing their order.
\end{remark}

\begin{remark}
\label{AlmostKey}
Let $\lambda$ be a partition and let $\alpha$ be a composition. If
$s \geq t$, then
\[
(K(\alpha))_{s} \subseteq (K(\alpha))_{t},
\]
as $K(\alpha)$ is a key tableau. Thus
\[
(B_{+}^{\lambda}(\alpha))_{s}
\subseteq
(B_{+}^{\lambda}(\alpha))_{t}
\]
as weakly decreasing words, since for each $r$, $(B_{+}^{\lambda}(\alpha))_{r}$ is obtained from $(K(\alpha))_{r}$ by applying the $\lambda$-operation $\lambda[\cdot]$ entrywise and then reordering the resulting entries.
\end{remark}

\begin{definition}
Let $c$ be a weak column, written from bottom to top as $c_1\geq\cdots\geq c_k$.
We define the strictification of $c$ as 
$$
d_{1}=c_{1}, \qquad  d_{i} = \min \{ d_{i-1}-1,c_{i} \}, \quad (1 < i \leq k).
$$   
Then $d_{1}>d_{2}>\cdots>d_{k}$ and $d\leq c$. We write
$\operatorname{str}(c)=d$.
\end{definition}
\begin{proposition}
\label{PropStrictification}
    The strictification of $c$ is the largest strictly increasing column bounded by $c$.
\end{proposition}
\begin{proof}
Indeed, let $a$ be a strictly increasing column, written from bottom to top as $a_{1}>\cdots>a_{k}$, and suppose that $a\leq c$.
Then $a_{1} \leq c_{1} = d_{1}$, and proceeding by induction:
$$
a_{i} \leq \min \{ c_{i}, a_{i-1}-1 \}  \leq \min \{ c_{i}, d_{i-1}-1\}.
$$
Therefore $a \leq d$. Note that if $d_k\leq0$, then no strictly increasing column with positive entries can be bounded by $c$.
\end{proof}

\begin{proposition}
\label{PropStrInclusion}
Let $v$ and $w$ be weakly decreasing words, with repeated occurrences counted with multiplicity. If $v\subseteq w$, then
\[
\operatorname{str}(v)\subseteq\operatorname{str}(w).
\]
In particular, if \(\operatorname{str}(w)\) has only positive entries,
then so does \(\operatorname{str}(v)\).
\end{proposition}

\begin{proof}
For a finite set \(S\subset\mathbb{Z}\) and an integer \(x\), define
\[
p_S(x)=\max\{z\in\mathbb{Z}:z\leq x,\ z\notin S\}.
\]
We interpret the operation $S\longmapsto S\cup\{p_S(x)\}$ as letting $x$ occupy the largest unoccupied position not exceeding $x$.

We first observe that the final set of occupied positions is independent of the order in which the letters are processed.
It suffices to show that the operations corresponding to two letters \(x\leq y\) commute.

Set $a=p_S(x)$ and $b=p_S(y)$. If $b>x$, then occupying $b$ does not affect the position $a$, while occupying $a$ does not affect the position $b$.
Thus both orders produce $S\cup\{a,b\}$. Suppose instead that \(b\leq x\).
Since \(x\leq y\), we must have
\(a=b\). Let
\[
c=\max\{z<a:z\notin S\}.
\]
After one of the two letters occupies \(a\), the other occupies \(c\).
Thus both orders produce $S\cup\{a,c\}$.
The two operations therefore commute, so the final occupied set is independent of the order of the letters.

Processing the letters of a weakly decreasing word $u=u_1\cdots u_m$ from left to right, strictification is given by
\[
\operatorname{str}(u)_1=u_1,
\qquad
\operatorname{str}(u)_i
=
\min\{u_i,\operatorname{str}(u)_{i-1}-1\},
\qquad 2\leq i\leq m.
\]

Now suppose that \(v\subseteq w\). Since the occupied-position construction is independent of the order, we may process first the occurrences belonging to \(v\), and then the remaining occurrences of \(w\).
After processing the letters of \(v\), the occupied positions are exactly the entries of \(\operatorname{str}(v)\).
Every remaining letter occupies a new position and does not vacate any previously occupied position.
Consequently, $\operatorname{str}(v)\subseteq\operatorname{str}(w)$.

Since both strictified words are strictly decreasing, containment as sets is equivalent to containment as subwords.
\end{proof}

\begin{remark}
  Since for $s \leq t$, $(B_{+}^{\lambda}(\alpha))_{s} \supseteq (B_{+}^{\lambda}(\alpha))_{t}$, by the previous proposition we have that $\operatorname{str}(B_{+}^{\lambda}(\alpha))_{s} \supseteq \operatorname{str}(B_{+}^{\lambda}(\alpha))_{t}$.
  Hence when $\operatorname{str}(B_{+}^{\lambda}(\alpha))_{1}$ is positive, $\operatorname{str}(B_{+}^{\lambda}(\alpha))_{t}$ is positive for all $t$.  
\end{remark}

\begin{definition}
\label{DefKeylambda}
Let $\lambda=(\lambda_1,\ldots,\lambda_n)$ be a partition, set $m=\lambda_1$, and let $\alpha\in\mathbb N^n$.

Suppose that the columnwise strictification $\operatorname{str}\bigl(B_+^\lambda(\alpha)\bigr)$ has only positive entries, where this condition is understood vacuously when $\alpha=0$. By the previous remark, this strictification is a key tableau.
We define $\alpha^\lambda\in\mathbb N^m$ to be the unique weak composition such that
\[
K(\alpha^\lambda)
=
\operatorname{str}\bigl(B_+^\lambda(\alpha)\bigr).
\]

If the strictification contains a nonpositive entry, we say that $\alpha^\lambda$ does not exist.
Finally, set
\[
\operatorname{Comp}(\lambda)
=
\left\{
\alpha\in\mathbb N^n:
\alpha^\lambda\text{ exists}
\right\}.
\]
\end{definition}

\begin{example}
Let
\[
\lambda=(4,4,3,1)
\qquad\text{and}\qquad
\alpha=(1,4,5,2).
\]
Then
\[
\alpha^+=(5,4,2,1).
\]
Hence
\[
K(\alpha)=
\tableau[scY]{
1 & 2 &2 &2 & 3 \\
2 & 3 & 3 & 3\\
3 & 4 \\
4
}.
\]
Applying the definition of $B_+^\lambda(\alpha)$ column by column, and reading the entries of each column from bottom to top, we obtain
\[
B_+^\lambda(\alpha)
=
\tableau[scY]{
1 & 1 &3 &3 & 3 \\
3 & 3 & 4 & 4\\
4 & 4 \\
4
}.
\]
Notice that $B_+^\lambda(\alpha)$ is not strictly increasing in its first column and hence is not a semistandard tableau.
Its role is to encode the $\lambda$-dependent upper bounds appearing in the admissibility condition.
Finally, applying strictification gives
\[
K(\alpha^\lambda)
=
\tableau[scY]{
1 & 1 &3 &3 & 3 \\
2 & 3 & 4 & 4\\
3 & 4 \\
4
}.
\]
\end{example}

\begin{example}
    If $\lambda = (1,1)$ and $\alpha =(1,1)$ then
$$K(\alpha) =\tableau[scY]{
1 \\
2 
}, \qquad
B_{+}^{\lambda}(\alpha) =\tableau[scY]{
1 \\
1 
}
\quad \text{ and } \quad  \operatorname{str}( B_{+}^{\lambda}(\alpha) )_{1} = \tableau[scY]{
0 \\
1 
}
.
$$
As $\operatorname{str}( B_{+}^{\lambda}(\alpha) )_{1}$ has a nonpositive
entry, $\alpha^{\lambda}$ does not exist.
\end{example}

Although the definition of $\alpha^\lambda$ is expressed in terms of key tableaux and strictification, $\alpha^\lambda$ can also be computed directly from $\alpha$ and the boundary values of $\lambda$.
We give this construction, together with an explicit criterion for the existence of $\alpha^\lambda$, in Section~\ref{SecDirectAlphaLambda}.

We now use strictification to replace the weak $\lambda$-dependent bound by an ordinary key bound. Throughout the remainder of this section, let $\lambda=(\lambda_1,\ldots,\lambda_n)$ and set $m=\lambda_1$.

\begin{definition}
\label{DefSBounds}
For \(\alpha\in\mathbb N^n\), define
\[
\mathcal{S}(\alpha)
=
\left\{
Q\in\operatorname{Tab}_{[n]}(\alpha^+):
K_+(Q)=K(\alpha)
\right\},
\]
and
\[
\mathcal{S}^\lambda(\alpha)
=
\left\{
P\in\operatorname{Tab}_{[m]}(\alpha^+):
K_+(P)\leq B_+^\lambda(\alpha)
\right\}.
\]
\end{definition}

\begin{proposition}
\label{PropStrictificationBound}
Let $\lambda=(\lambda_1,\ldots,\lambda_n)$, set $m=\lambda_1$,
and let $\alpha\in\mathbb N^n$.

If $\alpha\in\operatorname{Comp}(\lambda)$, then
\[
\mathcal{S}^\lambda(\alpha)
=
\left\{
P\in\operatorname{Tab}_{[m]}(\alpha^+):
K_+(P)\leq K(\alpha^\lambda)
\right\}.
\]
If $\alpha\notin\operatorname{Comp}(\lambda)$, then $\mathcal S^\lambda(\alpha)=\varnothing$.
\end{proposition}

\begin{proof}
Let $b$ be a column of $B_+^\lambda(\alpha)$. By Proposition~\ref{PropStrictification}, $\operatorname{str}(b)$ is the largest strictly increasing column bounded by $b$.
Consequently, for every key tableau $K$ of shape $\alpha^+$,
\[
K\leq B_+^\lambda(\alpha)
\quad\Longleftrightarrow\quad
K\leq \operatorname{str}\bigl(B_+^\lambda(\alpha)\bigr),
\]
where strictification is applied columnwise.

The columns of $B_+^\lambda(\alpha)$ are nested as weak columns.
By Proposition~\ref{PropStrInclusion}, their strictifications are also nested.
Hence, if the strictification of the first column has only positive entries, the strictified filling is the key tableau $K(\alpha^\lambda)$.
Therefore
\[
K_+(P)\leq B_+^\lambda(\alpha)
\quad\Longleftrightarrow\quad
K_+(P)\leq K(\alpha^\lambda).
\]

If the strictification of the first column contains a nonpositive entry, no strictly increasing column with positive entries can be bounded by that column.
Thus $\mathcal{S}^\lambda(\alpha)=\varnothing$.
\end{proof}

By Proposition~\ref{PropStrictificationBound} and the tableau formula~\eqref{EqKeyTableau}, we obtain the following consequence.

\begin{corollary}
\label{CorGeneratingSlambda}
Let $\lambda=(\lambda_1,\ldots,\lambda_n)$, set $m=\lambda_1$, and let $\alpha \in \mathbb{N}^{n}$.
Then
\[
\sum_{P\in \mathcal{S}^\lambda(\alpha)}y^{\operatorname{wt}(P)}
=
\begin{cases}
K_{\alpha^\lambda}(y_1,\ldots,y_m),
   & \text{if $\alpha^\lambda$ is defined},\\
0, & \text{otherwise}.
\end{cases}
\]
\end{corollary}

We now fix some notation that will be used throughout the section.
For a tableau $T$, recall that $\mathcal C_r(T)$ denotes the decreasing word corresponding to column $r$ of the right key $K_+(T)$.
Thus, if
\[
K_+(T)_r=
\begin{array}{c}
a_1\\
\vdots\\
a_{c_r}
\end{array},
\qquad
a_1\leq\cdots\leq a_{c_r},
\]
then
\[
\mathcal C_r(T)=a_{c_r}\cdots a_1.
\]

Following~\cite[Section~5.1]{mSchurt0}, we define $\mathcal C_r^\lambda(Q)$ as the $\lambda$-analogue of $\mathcal C_r^*(Q)$.
For a tableau $Q\in\operatorname{Tab}_{[n]}$, let $\mathcal{C}^{\lambda}_{r}(Q)$ be the word obtained by applying the $\lambda$-operation to the letters of $\mathcal{C}_{r}(Q)$ and then reordering the resulting letters so that the word is weakly decreasing.

For a tableau $Q$, we write
\[
\mathcal{C}_r^\lambda(Q)
=
\bigl(\mathcal{C}_r(Q)\bigr)^\lambda.
\]
By construction, $\mathcal C_r^\lambda(Q)$ is obtained by applying
the $\lambda$-operation to column $r$ of $K_+(Q)$.
In the notation of \cite{mSchurt0}, for $\lambda=(N+m,\ldots,N+1,N^{N-m})$, the word $\mathcal C_r^\lambda(Q)$ specializes to $\mathcal C_r^*(Q)$ defined in~\cite{mSchurt0} under the identification $N+j \leftrightarrow \hat{j}$.

The insertion analysis in Sections~\ref{Sec:RSK}--\ref{sec:RSKconsequences} extends the strategy developed in joint work with Lapointe~\cite[Sections~5.2--5.4]{mSchurt0}.
In that setting, the boundary is the particular Ferrers shape associated with the $m$-symmetric Cauchy kernel.
For a general partition \(\lambda\), the same local insertion mechanism survives, but repeated parts of \(\lambda\) produce weak column bounds.
This necessitates the strictification procedure introduced above and leads to the additional existence problem treated in Section~\ref{SecDirectAlphaLambda}.

\subsection{The RSK algorithm and $\lambda$-admissible pairs}

\label{Sec:RSK}

We now come to the technical core of the article: the compatibility between admissibility and a single step of the RSK insertion.
Recall from Section~\ref{ssectab} the row insertion algorithm $T \leftarrow j$ and the associated notion of insertion path.

Throughout this subsection, we retain the notation of Definition~\ref{DefSBounds}; in particular, $\lambda=(\lambda_1,\ldots,\lambda_n)$ and $m=\lambda_1$.

We say that $(P',Q')=(P,Q) \leftarrow \binom{i}{j}$ is RSK-compatible if the following conditions are satisfied

\begin{itemize}
\item $(P,Q)$ and $(P',Q')$ are pairs of tableaux of the same shape in $\operatorname{Tab}_{[m]} \times \operatorname{Tab}_{[n]}$.
\item $P'=P \leftarrow j$, the insertion of the letter $j$ in $P$.
\item  $Q$ is a subtableau of $Q'$, and the only cell in  $Q'/Q$ is filled with the letter $i$.
\item No letter in $Q$ is larger than $i$, and no letter $i$ in $Q$ occurs in a column to the right of the column in which $Q'/Q$ lies.
\end{itemize}

The elements of
\begin{equation}\label{eqGlambda}
\mathcal{G}_\lambda
=
\bigcup_{\alpha \in \operatorname{Comp}(\lambda)}
\left(
\mathcal{S}^{\lambda}(\alpha)
\times
\mathcal{S}(\alpha)
\right)
\end{equation}
will play a fundamental role in the proof of the Cauchy identity.
We say that a pair of tableaux $(P,Q)$ is \emph{$\lambda$-admissible} if $(P,Q)\in \mathcal{G}_\lambda$.
Equivalently, if $Q\in \mathcal{S}(\alpha)$, then
\[
P\in \mathcal{S}^\lambda(\alpha)
\quad\Longleftrightarrow\quad
(P,Q)\in \mathcal{G}_\lambda.
\]
Indeed, $K_+(Q)=K(\alpha)$, so the defining bound depends only on $\alpha$ and not on the particular choice of $Q\in \mathcal{S}(\alpha)$.

We say that $\binom{i}{j}$ is a $\lambda$-admissible biletter if $j \leq \lambda [i]$.
A $\lambda$-admissible biword is a biword of the form
\[
\left(
\begin{array}{cccc}
i_1 & i_2 & \cdots & i_r \\
j_1 & j_2 & \cdots & j_r
\end{array}
\right)
\]
in which every biletter $\binom{i_k}{j_k}$ is $\lambda$-admissible.

We denote by $\mathcal{B}_{\lambda}$ the set of all lexicographically ordered $\lambda$-admissible biwords.

\begin{theorem}
\label{TeoremaKeysLambda}
The RSK correspondence, when restricted to $\mathcal{B}_\lambda$, provides a bijection
\[
\operatorname{RSK}\colon \mathcal{B}_\lambda\longrightarrow \mathcal{G}_\lambda.
\]
\end{theorem}
We prove the theorem later in this section.
Its proof is based on the following stepwise equivalence:
\begin{equation}
\label{equivalencelambda}
\left[
(P,Q)\text{ and }\binom{i}{j}
\text{ are both $\lambda$-admissible}
\right]
\iff
(P',Q')\text{ is $\lambda$-admissible}.
\end{equation}

Before proving this equivalence, we establish the necessary local results.

The following characterizations of $\lambda$-admissibility will be used repeatedly.
The third one, which expresses the condition in terms of suprema of decreasing subwords, is the form in which the condition can be propagated through a single RSK insertion.

\begin{definition}
For a tableau $Q$, let $B_+^\lambda(Q)$ denote the filling obtained by applying the $\lambda$-operation to every column of $K_+(Q)$.
Thus, the word obtained by reading column $r$ of $B_+^\lambda(Q)$ from bottom to top is $\mathcal{C}_r^\lambda(Q)$. If $Q\in \mathcal{S}(\alpha)$, then
\[
B_+^\lambda(Q)=B_+^\lambda(\alpha).
\]
\end{definition}

\begin{proposition}\label{propcriterialambda}
Let \(P\) and \(Q\) be tableaux of the same shape, let \(\ell\) be their number of columns, and let \(c_r\) denote the length of their \(r\)-th column.
The following statements are equivalent to the $\lambda$-admissibility of $(P,Q)$:
\begin{enumerate}
    \item
    \[
    K_+(P)\leq
    B_+^\lambda(Q),
    \]

    \item
    \[
    \bigl(\mathcal C_r(P)\bigr)_{i}
    \leq
    \lambda\left[
        \bigl(\mathcal C_r(Q)\bigr)_{c_{r}+1-i} \right] \qquad \text{ for  every } r \text{ and }i=1,\ldots,c_r,
    \]

    \item
    \[
    \sup W(P_r)
    \leq
    \mathcal{C}_r^\lambda(Q) \qquad \text{ for every } 1 \leq r \leq \ell
    \]
\end{enumerate}
\end{proposition}

\begin{proof}
The equivalence of {\it 1} and $\lambda$-admissibility follows directly from the definition.
Indeed, by the definition of $B_+^\lambda(Q)$,
\[
\bigl( B_+^\lambda(Q)_r\bigr)_i
=
\lambda\left[
\bigl(K_+(Q)_r\bigr)_{c_r+1-i}
\right],
\]
so that the tableau inequality in {\it 1} is precisely the $\lambda$-admissibility condition.

Statement {\it 2} is simply the columnwise formulation of {\it 1}.
Finally, by Proposition~\ref{IgualdadKeys},
\[
\sup W(P_r)=\mathcal C_r(P).
\]
Hence the condition in {\it 2} can be rewritten in terms of the corresponding decreasing subwords, giving {\it 3}.
\end{proof}

\subsection{Behavior under a single RSK insertion}

This subsection adapts the insertion arguments of \cite[Sections~5.2--5.3]{mSchurt0} to the $\lambda$-dependent bounds introduced above.
Some results are recalled for the reader's convenience, while others are reformulated for the general boundary $\lambda$, with proofs included where an adaptation is needed.

Proposition~\ref{PropInsPQlambda} proves the forward implication in~\eqref{equivalencelambda}.
For the converse, we show that neither a $\lambda$-inadmissible initial pair nor a $\lambda$-inadmissible inserted biletter can produce an admissible pair; these are Propositions~\ref{PropReciprocalambda} and~\ref{LimiteBuenoYMalolambda}, respectively.

\begin{proposition}
  \label{PropInsercion}
Suppose that  $Q$ is a tableau all of whose entries are at most $i$.
Let $Q'$ be the tableau obtained from $Q$ by adding a letter  $i$ in column $\ell$ (assuming that this is possible), where $\ell$ is such that there is no letter $i$ to the right of column $\ell$ in $Q$.
We then have that
\begin{equation}
\label{InsercionRecording}
 {\mathcal C}_r(Q')
=
\left\{
 \begin{array}{ll}
   {\mathcal C}_r(Q) &\quad {\rm if~} r > \ell\\
   i{\mathcal C}_r(Q) &\quad {\rm if~} r =\ell \\ 
   {\mathcal C}_r'  & \quad  {\rm  if~ } r < \ell
\end{array}
\right.,
\end{equation}
where $ {\mathcal C}_r'$ is obtained from ${\mathcal C}_r(Q)$ by replacing by $i$ the largest entry of  ${\mathcal C}_r(Q)$ not in $ {\mathcal C}_\ell(Q)$ and then reordering the letters to get a decreasing word (observe that this amounts to doing nothing if this largest entry is $i$). 
\end{proposition}

\begin{proof}
    See~\cite[Proposition~39]{mSchurt0}.
\end{proof}

\begin{corollary}
\label{DesigualdadQlambda}
The effect on $\mathcal{C}^{\lambda}_{r}(Q')$ is
\[
\mathcal{C}^{\lambda}_{r}(Q')
=
\left\{
\begin{array}{ll}
\mathcal{C}_r^\lambda(Q)
&\quad \text{if } r>\ell,\\[2mm]
\mathcal{C}_r^\lambda(Q) \lambda[i]
&\quad \text{if } r=\ell,\\[2mm]
\mathcal{C}'_r
&\quad \text{if } r<\ell,
\end{array}
\right.
\]
where $\mathcal{C}_r'$ is obtained from $\mathcal{C}_r^\lambda(Q)$ by identifying the smallest letter remaining after deleting the occurrences of $\mathcal{C}_\ell^\lambda(Q)$, with multiplicities counted, replacing that occurrence by $\lambda[i]$, and then reordering the resulting word in weakly decreasing order.
\end{corollary}

\begin{proof}
Apply the weakly decreasing map $\lambda$ entrywise to the three cases of Proposition~\ref{PropInsercion}, and reorder the resulting words in weakly decreasing order.
\end{proof}

\begin{example}
Let
\[
\lambda=(7,5,5,4,4)
\qquad\text{and}\qquad
Q=
\tableau[scY]{
1 & 3 & 3\\
2 & 4 & 4\\
3\\
4
}.
\]
Add the letter \(i=5\) in column \(\ell=2\), obtaining
\[
Q'=
\tableau[scY]{
1 & 3 & 3\\
2 & 4 & 4\\
3 & 5\\
4
}.
\]
Since \(Q\) is already a key tableau, \(K_+(Q)=Q\). Moreover,
\[
K_+(Q')
=
\tableau[scY]{
1 & 3 & 3\\
3 & 4 & 4\\
4 & 5\\
5
}.
\]
In particular,
\[
\mathcal C_1(Q)=4321,\qquad
\mathcal C_2(Q)=43.
\]
and
\[
\mathcal C_1(Q')=5431,\qquad
\mathcal C_2(Q')=543.
\]

Now take \(r=1<\ell=2\). Since
\[
\mathcal C_r^\lambda(Q)=7\,5\,5\,4,
\qquad
\mathcal C_\ell^\lambda(Q)=5\,4,
\qquad
\lambda[i]=\lambda[5]=4,
\]
removing the occurrences of
\(\mathcal C_\ell^\lambda(Q)\) from
\(\mathcal C_r^\lambda(Q)\) leaves the letters \(7\) and \(5\).
Thus the smallest remaining letter is \(5\).
Replacing this occurrence of \(5\) by \(4\) and reordering gives
\[
\mathcal C_r^\lambda(Q')
=
7\,5\,4\,4.
\]
Notice that only one of the two occurrences of \(5\) is removed when
forming the difference, which is why multiplicities must be taken into
account.
\end{example}

\begin{corollary}
If $Q$ is a tableau that satisfies the conditions of
Proposition~\ref{PropInsercion}, then
\[
\mathcal{C}_r^\lambda(Q')
\leq
\mathcal{C}_r^\lambda(Q)
\qquad\text{for all } r\neq\ell.
\]
\end{corollary}

\begin{proof}
For $r>\ell$, the result follows from the equality $\mathcal{C}_r(Q')=\mathcal{C}_r(Q)$.
If $r<\ell$, then $\mathcal{C}_r(Q')$ is obtained from $\mathcal{C}_r(Q)$ by replacing a letter $q\leq i$ by $i$.
Since $\lambda$ is weakly decreasing, $\lambda[i]\leq\lambda[q]$.
After reordering, this gives
\[
\mathcal{C}_r^\lambda(Q')\leq\mathcal{C}_r^\lambda(Q).
\]
\end{proof}

\begin{remark}
In every RSK-compatible step
\[
(P',Q')=(P,Q)\leftarrow\binom{i}{j},
\]
the recording tableaux $Q$ and $Q'$ satisfy the hypotheses of Proposition~\ref{PropInsercion}.
\end{remark}

\begin{proposition} \label{propvuj}
Let  $c_{k-1}, \ldots , c_{2}, c_{1}, c_{0} (=j)$
be the insertion path of  $P'=P \leftarrow j$.
Given a word $w=w_t\cdots w_1  \in W(P_r')$, let $p$ be the largest integer such that $w_p$ intersects the insertion path (if there is no such integer, then $p=1$).
Let $v=w_t\cdots w_p v_{p-1} \cdots v_1 \in W(P'_r)$, where $v_{p-1} \cdots v_1$ is defined in the following way:
if $w_i$ lies in row $s+1$ weakly to the left (in the same row) of $c_{s}$, then let $v_i=c_{s}$.
Otherwise, let $v_i=w_i$. 
We then have that  $w\leq v$,  with  either $v \in W(P_r)$ or $v=uj$ for some $u \in W(P_r)$.
\end{proposition}  
\begin{proof}
This is precisely~\cite[Proposition~46]{mSchurt0}, in the present notation.
\end{proof}

\begin{example}
 If we insert the letter $j=3$ in  
 $$
 P=   \tableau[scY]{  1 & 2 & 4 & 5 \\  3  \\  6   
} 
 $$
then the insertion path (highlighted in blue) is given by
$$
P' = \tableau[scY]{  1 & 2 & {\color{blue} 3} & 5 \\  3 & {\color{blue} 4}  \\  6   
}
$$
Now, consider $w=42 \in W(P_{2}')$ as highlighted in red in $\tableau[scY]{  1 & {\color{red} 2} & 3 & 5 \\  3 & {\color{red} 4}  \\ 6 }$.
The procedure in Proposition~\ref{propvuj} then gives us $v=43 \in W (P_{2}')$ as highlighted in red in $\tableau[scY]{  1 & 2 & {\color{red} 3} & 5 \\  3 & {\color{red} 4}  \\ 6 }$, where we see that $v=uj$ with $u=4 \in W(P_{2})$ and $j=3$.

If we had chosen instead  $w=632  \in W(P_{1}')$ as highlighted in red in  $\tableau[scY]{  1 & {\color{red} 2} & 3 & 5 \\  {\color{red} 3} & 4  \\ {\color{red} 6}     }$, which does not intersect the insertion path, the procedure would then have given us $v = w \in W (P_{1}')$, with $v \in W ( P_{1} )$.
\end{example}

The next proposition adapts \cite[Proposition~50]{mSchurt0} to the present $\lambda$-dependent boundary.
\begin{proposition}
\label{PropInsPQlambda}
Let $(P',Q')=(P,Q)\leftarrow \binom{i}{j}$ be RSK-compatible.
If $(P,Q)$ and $\binom{i}{j}$ are both $\lambda$-admissible, then $(P',Q')$ is also a $\lambda$-admissible pair.
\end{proposition}

\begin{proof}

Let the new cell lie in row $k$ and column $\ell$. For each $r$, let $c_r$ be the length of column $r$ of $P'$. We must show that every $w\in W(P'_r)$ satisfies $w\leq\mathcal C_r^\lambda(Q')$.

The admissibility of $(P,Q)$ says that $\sup  W(P_r) \leq  {\mathcal C}_r^\lambda(Q)$ (see Proposition~\ref{propcriterialambda}), so that every $w \in  W(P_r)$ satisfies $w \leq {\mathcal C}_r^\lambda(Q)$.
Additionally, we will use $v$ with $w \leq v$ as in Proposition~\ref{propvuj}. Since every entry of $Q'$ is at most $i$ and $\lambda$ is weakly decreasing, every letter of $\mathcal C_r^\lambda(Q)$ and $\mathcal C_r^\lambda(Q')$ is at least $\lambda[i]$.

{\it The case $r > \ell$.} Corollary~\ref{DesigualdadQlambda} gives here ${\mathcal C}_r^\lambda(Q')= {\mathcal C}_r^\lambda(Q)$.
If $v \in W(P_r)$, the observation yields $w \leq v \leq {\mathcal C}_r^\lambda(Q)={\mathcal C}_r^\lambda(Q')$.
If instead $v=uj$, then replacing $v$ by $u$ in the previous argument shows that all the letters of $w$, except possibly its last letter, are bounded by $\mathcal{C}^{\lambda}_{r}(Q')$.
Denote the last letter of $w$ by $y$. Then
$$
y \leq j \leq \lambda [i] \leq (\mathcal{C}^{\lambda}_{r}(Q'))_{c_{r}} \leq (\mathcal{C}^{\lambda}_{r}(Q'))_{s}
$$
since $|w|=s \leq c_{r}$, $\mathcal{C}^{\lambda}_{r}(Q')$ is weakly decreasing and $j \leq \lambda [i]$ as $(i,j)$ is $\lambda$-admissible.

\smallskip
\noindent
{\it The case $r=\ell$.} Corollary~\ref{DesigualdadQlambda} gives here
  ${\mathcal C}_r^\lambda(Q')= {\mathcal C}_r^\lambda(Q) \lambda[i]$.
  If $v \in W(P_r)$, the observation yields
$$
w \leq v \leq  {\mathcal C}_r^\lambda(Q) \leq {\mathcal C}_r^\lambda(Q)  \lambda[i] = {\mathcal C}_r^\lambda(Q').
$$
If instead $v=uj$, then all the letters of $w$, except possibly its last letter, are bounded by $u\in W(P_r)$, and hence by $\mathcal{C}_r^\lambda(Q)$. Denote the last letter of $w$ by $y$. Then
$$
y \leq j \leq \lambda [i] = \min \mathcal{C}^{\lambda}_{r}(Q') \leq (\mathcal{C}^{\lambda}_{r}(Q'))_{s}.
$$
Therefore $w \leq \mathcal{C}^{\lambda}_{r}(Q')$.

\smallskip
\noindent
{\it The case $r<\ell$.}
Set $a=\lambda[i]$. We first show that $\mathcal C_r(P')\leq\mathcal C_r^\lambda(Q)$.
For $w\in W(P_r')$, Proposition~\ref{propvuj} gives
$w\leq v$, where either $v\in W(P_r)$ or $v=uj$
with $u\in W(P_r)$.
In the first case, admissibility gives
$v\leq\mathcal C_r^\lambda(Q)$.
In the second, the letters of $u$ satisfy the corresponding
bounds by admissibility, while $j\leq a\leq\min\mathcal C_r^\lambda(Q)$. Since $r<\ell$, column $r$ has unchanged length, so
$|v|=|w|\leq|\mathcal C_r^\lambda(Q)|$.
Thus $w\leq\mathcal C_r^\lambda(Q)$ in both cases.
Taking the supremum gives the claimed bound.

The word $\mathcal C_\ell^\lambda(Q)$ is a subword of
$\mathcal C_r^\lambda(Q)$, with occurrences counted with
multiplicity. Let $S$ be their longest common suffix, and write
\[
\mathcal C_\ell^\lambda(Q)=US,
\qquad
\mathcal C_r^\lambda(Q)=VbS.
\]
Here $b$ is the smallest remaining letter of
$\mathcal C_r^\lambda(Q)$ after removing the occurrences of
$\mathcal C_\ell^\lambda(Q)$. Corollary~\ref{DesigualdadQlambda} therefore gives $\mathcal C_r^\lambda(Q')=VSa$.
The bound $\mathcal C_r(P')\leq VbS$ already bounds the first
$|V|$ letters by $V$. It remains to bound the final $|S|+1$
letters by $Sa$.

Since $K_+(P')$ is a key tableau,
$\mathcal C_\ell(P')$ is a subword of $\mathcal C_r(P')$.
As both words are decreasing, the final $|S|+1$ letters of
$\mathcal C_r(P')$ are componentwise at most the final
$|S|+1$ letters of $\mathcal C_\ell(P')$.
The latter are bounded by $Sa$, because the case $r=\ell$
already gives
\[
\mathcal C_\ell(P')
\leq\mathcal C_\ell^\lambda(Q)a=USa,
\]
and both words have length $k$.
Combining the prefix and suffix bounds yields
\[
\mathcal C_r(P')\leq VSa=\mathcal C_r^\lambda(Q'),
\]
as required.
\end{proof}

For the reverse implication, we first record the following consequence of Lemma~\ref{SupKnuthEquivalent}.

\begin{lemma}
\label{LemaDesfasaje} Given
a tableau  $P$ and a letter $j$, let  $P'=P \leftarrow j$ be such that the insertion path ends in column $\ell$.
For all $r \leq \ell$ we have that $\sup W(P_r) \leq \sup W(P_r')$.
\end{lemma}
\begin{proof}
If $r \leq \ell$, then $w(P_{r})j$ is Knuth equivalent to $w(P_{r}')$.
Denote by $ W(P_{r})j$ the set of decreasing words $w=uj$ with $u \in  W(P_{r})$.
Since $W(P_{r}) \subseteq W(P_{r}) \cup  W(P_{r})j$, we then have from Lemma~\ref{SupKnuthEquivalent} that
$$
\sup W(P_r)  \leq \sup \bigl(  W(P_{r}) \cup  W(P_{r})j \bigr)= \sup W(P_r').
$$
\end{proof}

The backward implication requires the corresponding $\lambda$-dependent version of \cite[Proposition~53]{mSchurt0}.

\begin{proposition}
\label{PropReciprocalambda}
Let $(P',Q')=(P,Q)\leftarrow \binom{i}{j}$ be RSK-compatible.
If the pair $(P,Q)$ is not $\lambda$-admissible, then the
pair $(P',Q')$ is not $\lambda$-admissible.

\end{proposition}

\begin{proof}
Throughout this proof we use two indexing conventions for decreasing
words.

If $w=a_1\cdots a_s$, we write $(w)_t=a_t$ for the $t$-th
letter of $w$ from left to right.

In arguments involving insertion paths, we instead label the same word as $w=w_s\cdots w_1$, following the row-oriented convention used for the insertion path.
Thus
\[
(w)_t=w_{s+1-t},
\qquad 1\leq t\leq s,
\]
and in particular $(w)_s=w_1$.

Throughout the proof, let the cell in $Q'/Q$ lie in row $k$ and column $\ell$, and suppose that $\sup W(P_r)\not\leq\mathcal C_r^\lambda(Q)$ for some $r$.
We consider separately the three possible positions of $r$ relative to $\ell$.

\smallskip
\noindent
{\it The case $r<\ell$.}
Lemma~\ref{LemaDesfasaje} gives $\sup W(P_r)\leq \sup W(P_r')$.
Hence $\sup W(P_r')\not\leq \mathcal{C}_r^\lambda(Q') $ since otherwise Corollary~\ref{DesigualdadQlambda} would give
\[
\sup W(P_r)
\leq \sup W(P_r')
\leq \mathcal{C}_r^\lambda(Q')
\leq \mathcal{C}_r^\lambda(Q),
\]
a contradiction.

\smallskip
\noindent
{\it The case $r=\ell$.}
As before, $w(P_\ell')$ is Knuth equivalent to $w(P_\ell)j$, and therefore $\sup W(P_\ell)\leq \sup W(P_\ell')$ by Lemma~\ref{LemaDesfasaje}.
If we suppose that $\sup W(P_\ell')\leq \mathcal C_\ell^\lambda(Q')$, then we get from Corollary~\ref{DesigualdadQlambda} that
\[
\sup W(P_\ell)
\leq
\sup W(P_\ell')
\leq
\mathcal C_\ell^\lambda(Q')
=
\mathcal C_\ell^\lambda(Q)\,\lambda[i].
\]

The extra letter $\lambda[i]$ cannot repair the inequality $\sup W(P_\ell)\not\leq \mathcal C_\ell^\lambda(Q)$ since both words have the same length and $\lambda[i]$ is appended at the end.

Consequently, $\sup W(P_\ell')\not\leq
\mathcal C_\ell^\lambda(Q')$, as desired.

\smallskip
\noindent
{\it The case $r>\ell$.}
Since $\sup W(P_r)\not\leq \mathcal{C}_r^\lambda(Q)$, by the definition of the supremum, after truncating after the first $s$ letters if necessary, we may choose a decreasing word $
w=w_s\cdots w_1\in W(P_r)$ of length $s$ such that
$$
(w)_s=w_1=\bigl(\sup W(P_r)\bigr)_s 
>
\bigl(\mathcal{C}_r^\lambda(Q)\bigr)_s.
$$
Suppose first that $w$ does not intersect the insertion path.
Then $w\in W(P_r')$, and therefore
\[
\bigl(\sup W(P_r')\bigr)_s
\geq (w)_s
=
\bigl(\sup W(P_r)\bigr)_s
>
\bigl(\mathcal{C}_r^\lambda(Q)\bigr)_s
=
\bigl(\mathcal{C}_r^\lambda(Q')\bigr)_s
\]
by Corollary~\ref{DesigualdadQlambda}.
Hence $ \sup W(P_r')\not\leq\mathcal{C}_r^\lambda(Q') $ in that case.

Suppose now, on the contrary, that $w$ intersects the insertion path $c_{k-1} \ldots c_{1}$ in $P$.
Let $t$ be the smallest integer such that $w_{t}$  intersects the insertion path, and suppose that $w_t=c_{p-1}$.
Observe that the entry $w_t$ lies in row  $p-1$  in  $P$ while it lies in row $p$ in $P'$.
    
We consider the word $u=c_{k-1} \cdots c_p w_t \cdots w_1$ of length  $(k-p)+t$.
The word $u$ is decreasing since $c_{k-1} \cdots c_1$ is decreasing, $w$ is decreasing and $c_p > c_{p-1} =w_t$.
We also have that  $u \in W(P_\ell')$ since $w_t$ lies in the row above that of $c_p$ in $P'$ and since $w_{t-1},\dots, w_1$ belong to $P'$ given that they do not intersect the insertion path.

Suppose that $P_r$ has $q$ rows.
In rows $p+1,\ldots,q$ of $P'$ lie the decreasing words  $c_{q-1} \ldots c_p$ and $w_{s} \dots w_{t+1}$.
Since the former has a letter in each of those rows, we have that $ s-t \leq q-p $, or equivalently, that $k-q+s \leq k-p+t$.
Since $u$ is a decreasing word, we thus have that $ (u)_{k-q+s} \geq (u)_{k-p+t} =w_{1}=(w)_s$.
Because $u \in  W(P_\ell')$,  this means that   $(\sup  W(P_\ell') )_{k-q+s} \geq  (u)_{k-q+s} \geq (w)_s$.
Hence, Corollary~\ref{DesigualdadQlambda} yields that
$$
(\sup  W(P_\ell') )_{k-q+s}   \geq (w)_{s} >
 ( {\mathcal C}_r^{\lambda}(Q))_s   =  ( {\mathcal C}_r^{\lambda}(Q'))_s \geq  ( {\mathcal C}_\ell^{\lambda}(Q'))_{k-q+s},
$$
which implies that
$\sup W(P_\ell') \not \leq  {\mathcal C}_\ell^{\lambda}(Q')$.
Note that $ ( {\mathcal C}_r^{\lambda}(Q'))_s \geq  ( {\mathcal C}_\ell^{\lambda}(Q'))_{k-q+s}$
since ${\mathcal C}_r^{\lambda}(Q')$ is a subword of $ {\mathcal C}_\ell^{\lambda}(Q')$
and since the difference in length between the two words is $k-q$ (as
$\left|\mathcal C_\ell^\lambda(Q')\right|=k$ and $\left|\mathcal C_r^\lambda(Q')\right|=q$).
\end{proof}

\begin{example}
Let
\[
\lambda=(4,3,1,1),
\]
and consider
\[
(P,Q)
=
\left(
\tableau[scY]{
2 & 3\\
3
},
\tableau[scY]{
1 & 1\\
3
}
\right),
\qquad
(i,j)=(4,1).
\]
We have
\[
\mathcal C_1(Q)=31,
\qquad
\mathcal C_1^\lambda(Q)=41.
\]
On the other hand,
\[
32\in W(P_1)
\qquad\text{and}\qquad
32\not\leq41.
\]
Thus, $(P,Q)$ is not $\lambda$-admissible.

The biletter $\binom{4}{1}$ is $\lambda$-admissible, since
\[
1\leq\lambda[4]=1.
\]
Its insertion gives
\[
(P',Q')
=
\left(
\tableau[scY]{
{\color{blue}1} & 3\\
{\color{blue}2}\\
{\color{blue}3}
},
\tableau[scY]{
1 & 1\\
3\\
4
}
\right).
\]
Here the insertion path is highlighted in blue. We now have
\[
\mathcal C_1(Q')=431,
\qquad
\mathcal C_1^\lambda(Q')=411.
\]
The word $32$ is still present in $W(P_1')$, as highlighted in red in
\[
\tableau[scY]{
1 & 3\\
{\color{red}2}\\
{\color{red}3}
}.
\]
Since
\[
32\not\leq411,
\]
the pair $(P',Q')$ is not $\lambda$-admissible.
This illustrates the case $r=\ell=1$ in the proof of Proposition~\ref{PropReciprocalambda}.
\end{example}

\begin{example}
Let
\[
\lambda=(4,1,1,1),
\]
and consider
\[
(P,Q)
=
\left(
\tableau[scY]{
1 & 1\\
2
},
\tableau[scY]{
1 & 2\\
3
}
\right),
\qquad
(i,j)=(4,1).
\]
In this case,
\[
\mathcal C_1(Q)=32,
\qquad
\mathcal C_1^\lambda(Q)=11.
\]
However,
\[
21\in W(P_1)
\qquad\text{and}\qquad
21\not\leq11.
\]
Hence $(P,Q)$ is not $\lambda$-admissible.

Again, the biletter $\binom{4}{1}$ is $\lambda$-admissible because
\[
1\leq\lambda[4]=1.
\]
Its insertion gives
\[
(P',Q')
=
\left(
\tableau[scY]{
1 & 1 & {\color{blue}1}\\
2
},
\tableau[scY]{
1 & 2 & 4\\
3
}
\right).
\]
The new cell lies in column $\ell=3$, while the failure of admissibility occurs for $r=1<\ell$.
We have
\[
\mathcal C_1(Q')=42,
\qquad
\mathcal C_1^\lambda(Q')=11.
\]
Thus the repeated parts
\[
\lambda[2]=\lambda[3]=\lambda[4]=1
\]
cause the transformed bound to remain unchanged.
Moreover, the word $21$ is still present in $W(P_1')$, as highlighted in red in
\[
\tableau[scY]{
{\color{red}1} & 1 & 1\\
{\color{red}2}
}.
\]
Therefore,
\[
21\not\leq11,
\]
and $(P',Q')$ remains non-$\lambda$-admissible.
\end{example}

We finally need the $\lambda$-dependent analogue of \cite[Proposition~56]{mSchurt0}.
\begin{proposition}
  \label{LimiteBuenoYMalolambda}
 Let $(P',Q')=   (P,Q) \leftarrow \binom{i}{j}$ be RSK-compatible.
 If  the biletter $\binom{i}{j}$ is not $\lambda$-admissible then neither is the pair $(P',Q')$.
\end{proposition}
\begin{proof}
Let $c_{k-1},\ldots,c_1,c_0(=j)$ be the insertion path, where $\ell$ is the column of the cell in $Q'/Q$ and $k$ is the length of column $\ell$ of $Q'$.
The word $c_{k-1}\cdots c_{1}j$ is a decreasing word of length $k$ in $W(P_\ell')$.
Since $\binom{i}{j}$ is not $\lambda$-admissible, we have $j > \lambda [i]$.
Hence, $(P',Q')$ is not $\lambda$-admissible since $(\sup W(P_\ell') )_{k} \geq j > \lambda [i] = ( {\mathcal C}_\ell^{\lambda}(Q'))_k$ by Corollary~\ref{DesigualdadQlambda}.
\end{proof}

\subsection{The RSK correspondence and its consequences}
\label{sec:RSKconsequences}

Having established \eqref{equivalencelambda} one insertion at a time, we can now iterate the result and obtain the desired consequences.

The RSK correspondence provides a bijection between the set of all biwords with top rows in $\{1,\dots,n\}$ and bottom rows in $\{1,\dots,m\}$ and pairs of tableaux $(P,Q)$ of the same shape, where $Q$ has entries in $\{1,\dots,n\}$ and $P$ has entries in $\{1,\dots,m\}$. More specifically, given a biword
\[
\mathbf b=
\left(
\begin{array}{cccc}
i_1 & i_2 & \cdots & i_r \\
j_1 & j_2 & \cdots & j_r
\end{array}
\right)
\]
in lexicographic order, one constructs successively a pair
$(P,Q)=(P^{(r)},Q^{(r)})$ from the empty pair by applying
\[
(P^{(s)},Q^{(s)})
=
(P^{(s-1)},Q^{(s-1)})\leftarrow
\binom{i_s}{j_s}.
\]
The lexicographic ordering guarantees that each of these insertions is
RSK-compatible.

The preceding results can now be assembled into the main theorem.

\begin{proof}[Proof of Theorem~\ref{TeoremaKeysLambda}.]
\label{ProofTheoremlambda}
Let $(P',Q')=   (P,Q) \leftarrow \binom{i}{j}$ be RSK-compatible.
As mentioned earlier,  Propositions~\ref{PropInsPQlambda}, \ref{PropReciprocalambda} and \ref{LimiteBuenoYMalolambda} tell us that
\[
\left[
(P,Q)\text{ and }\binom{i}{j}
\text{ are both $\lambda$-admissible}
\right]
\iff
(P',Q')\text{ is $\lambda$-admissible}.
\]

Starting from a biword $\mathbf b\in \mathcal{B}_\lambda$ in lexicographic order and applying the RSK insertion one biletter at a time, the forward implication shows inductively that every intermediate pair is $\lambda$-admissible.
Hence the resulting pair belongs to $\mathcal{G}_\lambda$.
Conversely, let $(P,Q)\in \mathcal{G}_\lambda$.
Applying the inverse RSK correspondence recovers the biletters of the biword one at a time.
The reverse implication above shows inductively that each recovered biletter is $\lambda$-admissible.
Thus the recovered biword belongs to $\mathcal{B}_\lambda$.

Since RSK is a bijection before imposing the admissibility conditions,
these two implications show that its restriction gives the claimed bijection.
\end{proof}

\begin{corollary}
\label{CorKernellambda}
Let $\lambda=(\lambda_1,\ldots,\lambda_n)$ be a partition and set $m=\lambda_1$.
Then
\begin{equation}
\label{EqKernellambda}
\prod_{(i,j)\in\lambda}\frac{1}{1-x_i y_j}
=
\sum_{\alpha \in \operatorname{Comp}(\lambda)}
\hat K_\alpha(x)K_{\alpha^\lambda}(y)
=
\sum_{\begin{smallmatrix}
      \alpha\in \operatorname{Comp}(\lambda), \:\beta\in\mathbb{N}^m\\
      \beta^+=\alpha^+, \;
      K(\beta)\leq K(\alpha^\lambda)
      \end{smallmatrix}
      }
\hat K_\alpha(x)\hat K_\beta(y).
\end{equation}
\end{corollary}

\begin{proof}
By Theorem~\ref{TeoremaKeysLambda}, RSK restricts to a weight-preserving bijection between $\mathcal{B}_{\lambda}$ and $\mathcal{G}_{\lambda}$.
More precisely, if a biword $w$ has top and bottom weights $\pmb a$ and $\pmb b$, respectively, then its image $(P,Q)$ satisfies
$$
\operatorname{wt}(P)=\pmb b
\qquad\text{and}\qquad
\operatorname{wt}(Q)=\pmb a.
$$
Thus, both the biword $w$ and its image contribute the monomial $x^{\pmb a}y^{\pmb b}$.
$$
\sum_{w \in \mathcal{B}_{\lambda}} x^{\operatorname{wt}_{\operatorname{top}} (w)} y^{\operatorname{wt}_{\operatorname{bottom}} (w)}
=
\prod_{(i,j)\in\lambda}\frac{1}{1-x_i y_j}.
$$
On the other hand, using
$$
\mathcal{G}_\lambda=\bigcup_{\alpha \in \operatorname{Comp}(\lambda)}
\bigl(\mathcal{S}^\lambda(\alpha)\times \mathcal{S}(\alpha)\bigr),
$$
its generating function is
\begin{align*}
\sum_{(P,Q)\in \mathcal{G}_\lambda}
x^{\operatorname{wt}(Q)}y^{\operatorname{wt}(P)}
&=
\sum_{\alpha \in \operatorname{Comp}(\lambda)}
\left(\sum_{Q\in \mathcal{S}(\alpha)}x^{\operatorname{wt}(Q)}\right)
\left(\sum_{P\in \mathcal{S}^\lambda(\alpha)}y^{\operatorname{wt}(P)}\right)
=
\sum_{\alpha \in \operatorname{Comp}(\lambda)}\hat K_{\alpha}(x)K_{\alpha^\lambda}(y).
\end{align*}
Here, \eqref{EqAtomTableau} and Corollary~\ref{CorGeneratingSlambda} were used in the second equality.
Finally, expanding $K_{\alpha^\lambda}(y)$ by \eqref{EqKeyIntoAtoms} gives the last expression in \eqref{EqKernellambda}.
\end{proof}

\section{An admissibility criterion and a direct construction of $\alpha^\lambda$}
\label{SecDirectAlphaLambda}

The definition of $\alpha^\lambda$ uses the strictification of a tableau bound.
Here we give an existence criterion and a construction that can both be read directly from $\alpha$ and $\lambda$.
The criterion depends only on the support of $\alpha$; the construction places its nonzero parts by a parking procedure.

\begin{proposition}[Existence criterion for $\alpha^\lambda$]
\label{PropExistenceAlphaLambda}
Let $\lambda=(\lambda_1,\ldots,\lambda_n)$ be a partition, set
$m=\lambda_1$, and let $\alpha\in\mathbb N^n$.
If $\alpha=0^n$, then $\alpha^\lambda=0^m$.
Otherwise, write
\[
\operatorname{supp}(\alpha)=\{i_1<\cdots<i_c\}.
\]
Then $\alpha^\lambda$ exists if and only if
\[
\lambda[i_j]\geq c-j+1\qquad(1\leq j\leq c).
\]
Equivalently,
\[
\min_{1\leq j\leq c}\{\lambda[i_j]-(c-j)\}\geq1.
\]
\end{proposition}

\begin{proof}
The zero case follows from Definition~\ref{DefKeylambda}.
Suppose that $\alpha\neq0$.
Since the strictified columns of $B_+^\lambda(\alpha)$ are nested,
$\alpha^\lambda$ exists precisely when the first strictified column
has only positive entries.

The first column of $K(\alpha)$ consists of
$i_1<\cdots<i_c$. Thus the first column of $B_+^\lambda(\alpha)$,
read from bottom to top, is
\[
\lambda[i_1]\geq\cdots\geq\lambda[i_c].
\]
Write its strictification in the same order as $d_1>\cdots>d_c$.
The defining recurrence gives
\[
d_1=\lambda[i_1],\qquad
d_j=\min\{\lambda[i_j],d_{j-1}-1\}\quad(2\leq j\leq c),
\]
and hence
\[
d_c=\min_{1\leq j\leq c}\{\lambda[i_j]-(c-j)\}.
\]
All entries are positive if and only if the smallest one, $d_c$,
is positive. This is exactly the stated criterion.
\end{proof}

\begin{corollary}
\label{CorDistinctParts}
If $\lambda=(\lambda_1,\ldots,\lambda_n)$ has distinct positive parts,
then $\operatorname{Comp}(\lambda)=\mathbb N^n$.
\end{corollary}

\begin{proof}
For a nonzero composition, write
$\operatorname{supp}(\alpha)=\{i_1<\cdots<i_c\}$.
Distinctness of the parts gives $\lambda[i]\geq n-i+1$, while
$i_j\leq n-c+j$. Therefore
\[
\lambda[i_j]\geq n-i_j+1\geq c-j+1,
\]
so Proposition~\ref{PropExistenceAlphaLambda} applies.
The zero composition belongs to $\operatorname{Comp}(\lambda)$
by definition.
\end{proof}

\begin{example}
For $\lambda=(4,3,1,1)$ and $\alpha=(1,2,3,1)$, the support of
$\alpha$ is $\{1,2,3,4\}$.
The criterion fails at $j=3$, since
\[
\lambda[3]=1<2=4-3+1.
\]
Thus $\alpha^\lambda$ does not exist.
The same conclusion holds for every composition with this support. More generally, if $\alpha^\lambda$ exists, then $|\operatorname{supp}(\alpha)|$ is at most the number of rows of the largest staircase contained in $\lambda$.
\end{example}

We next construct $\alpha^\lambda$ by placing the nonzero parts of $\alpha$ into positions $1,\ldots,m$.
Each part starts at the end of its corresponding row of $\lambda$ and moves left until it reaches an unoccupied position.

\begin{proposition}[Parking construction of $\alpha^\lambda$]
\label{Parkingalpha}
Let $\lambda=(\lambda_1,\ldots,\lambda_n)$ be a partition, set $m=\lambda_1$, and let $\alpha\in\mathbb N^n$.
For $\alpha=0^n$, the procedure returns $0^m$.
Otherwise, order the support as $i_1,\ldots,i_c$ so that
\[
\alpha_{i_1}\geq\cdots\geq\alpha_{i_c}>0,
\]
breaking ties arbitrarily, and process the indices in this order.
At step $t$, choose
\[
q_t=\max\left(
\{1,\ldots,\lambda[i_t]\}\setminus\{q_1,\ldots,q_{t-1}\}
\right).
\]
If the set is empty, stop and declare failure.
Otherwise, place the part $\alpha_{i_t}$ in position $q_t$.
If all steps succeed, let $\beta\in\mathbb N^m$ be the composition given by
\[
\beta_{q_t}=\alpha_{i_t}\quad(1\leq t\leq c),
\qquad
\beta_j=0\quad\text{at every unoccupied position}.
\]
The procedure succeeds if and only if $\alpha^\lambda$ exists.
When it succeeds, $\beta=\alpha^\lambda$; in particular, the output is independent of the order chosen among equal parts of $\alpha$.
\end{proposition}

\begin{proof}
The zero case is immediate, so assume $\alpha\neq0$.
Temporarily allow parking at all integer positions, and write $\widetilde q_t$ for the position chosen at step $t$.
This extended procedure always succeeds and agrees with the positive procedure until the latter fails.
As shown in the proof of Proposition~\ref{PropStrInclusion}, its occupied set depends only on the multiset of bounds.
For weakly decreasing bounds, that set is precisely their strictification.

To recover the individual columns, set
\[
I_r(\alpha)=\{i:\alpha_i\geq r\},\qquad c_r=|I_r(\alpha)|.
\]
The ordering by decreasing parts ensures that
$I_r(\alpha)=\{i_1,\ldots,i_{c_r}\}$.
Column $r$ of $K(\alpha)$ has exactly these entries, so the corresponding bounds are exactly those processed in the first $c_r$
steps.
Identifying a strict column with its set of entries, we obtain
\begin{equation}
\label{EqParkingColumns}
\operatorname{str}\bigl(B_+^\lambda(\alpha)_r\bigr)
=\{\widetilde q_t:1\leq t\leq c_r\}.
\end{equation}

For $r=1$, we have $c_1=c$.
Thus the positive procedure succeeds if and only if the first strictified column has only positive entries.
By Definition~\ref{DefKeylambda} and the nesting of the strictified columns, this is equivalent to the existence of $\alpha^\lambda$.

Suppose now that the procedure succeeds, so
$q_t=\widetilde q_t$ for every $t$.
The output $\beta$ satisfies, for every $r\geq1$,
\[
\{j:\beta_j\geq r\}
=\{q_t:\alpha_{i_t}\geq r\}
=\{q_t:1\leq t\leq c_r\}.
\]
The left-hand side is the set of entries in column $r$ of $K(\beta)$; by~\eqref{EqParkingColumns}, the right-hand side is the set of entries in column $r$ of $\operatorname{str}(B_+^\lambda(\alpha))$.
Consequently,
\[
K(\beta)=\operatorname{str}\bigl(B_+^\lambda(\alpha)\bigr)
=K(\alpha^\lambda),
\]
and hence $\beta=\alpha^\lambda$.
The argument applies to every ordering compatible with the parts of $\alpha$, proving independence of the choices made when breaking ties.
\end{proof}

\begin{example}
Let
\[
\lambda=(8,8,7,6,6,4,3)
\qquad\text{and}\qquad
\alpha=(3,7,4,6,2,5,1).
\]
The parts of $\alpha$ are processed in decreasing order, corresponding to the indices
\[
i_1,\ldots,i_7=2,4,6,3,1,5,7.
\]

The following diagram records the parking procedure directly on the Ferrers diagram of $\lambda$.
In row $i$, the car $\alpha_i$ is placed in its final parking position.
A dot marks its initial position $\lambda[i]$, and an arrow is drawn when the car must move to the first available position on its left.
The labels represent cars and are not tableau entries.

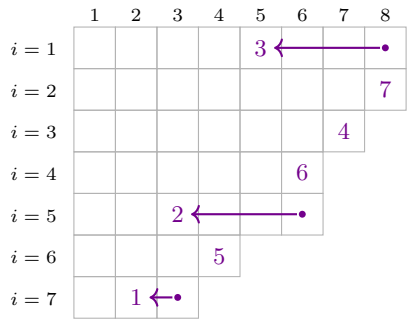
\begin{center}
\colorlet{parkingcolor}{blue!55!red}
\begin{tikzpicture}[
  x=.55cm,
  y=.55cm,
  car/.style={
    font=\small,
    text=parkingcolor,
    fill=white,
    inner sep=1pt
  }
]

\foreach \i/\len in {
  1/8,
  2/8,
  3/7,
  4/6,
  5/6,
  6/4,
  7/3
}{
  \node[anchor=e,font=\scriptsize]
    at (-.18,.5-\i) {$i=\i$};

  \foreach \j in {1,...,\len}{
    \draw[gray!60]
      (\j-1,1-\i) rectangle (\j,-\i);
  }
}

\foreach \j in {1,...,8}{
  \node[font=\scriptsize] at (\j-.5,.28) {$\j$};
}

\node[car] at (7.5,-1.5) {$7$};
\node[car] at (6.5,-2.5) {$4$};
\node[car] at (5.5,-3.5) {$6$};
\node[car] at (3.5,-5.5) {$5$};

\fill[parkingcolor] (7.5,-.5) circle (1.3pt);
\draw[
  parkingcolor,
  ->,
  thick,
  shorten <=2pt,
  shorten >=5pt
]
  (7.5,-.5) -- (4.5,-.5);
\node[car] at (4.5,-.5) {$3$};

\fill[parkingcolor] (5.5,-4.5) circle (1.3pt);
\draw[
  parkingcolor,
  ->,
  thick,
  shorten <=2pt,
  shorten >=5pt
]
  (5.5,-4.5) -- (2.5,-4.5);
\node[car] at (2.5,-4.5) {$2$};

\fill[parkingcolor] (2.5,-6.5) circle (1.3pt);
\draw[
  parkingcolor,
  ->,
  thick,
  shorten <=2pt,
  shorten >=5pt
]
  (2.5,-6.5) -- (1.5,-6.5);
\node[car] at (1.5,-6.5) {$1$};
\end{tikzpicture}
\end{center}

The cars $7,6,5$, and $4$ park at their initial bounds.
The car $3$ moves from position $8$ to position $5$, the car $2$ moves from position $6$ to position $3$, and the car $1$ moves from position $3$ to position $2$.

Reading the parked cars by position, from left to right, and placing $0$ in the unoccupied first position, gives
\[
\alpha^\lambda=(0,1,2,5,3,6,4,7).
\]
\end{example}

\section{Skew Ferrers shapes and two-sided key bounds}
\label{SecSkewFerrers}

The preceding results characterize the RSK image of matrices supported on a Ferrers diagram $\lambda$.
We now extend this description to a skew Ferrers diagram $\lambda/\mu$.
The outer boundary $\lambda$ produces an upper bound on the right key of the insertion tableau, whereas the inner boundary $\mu$ produces, through Schützenberger evacuation, a lower bound on its left key.

Let $\mu\subseteq\lambda\subseteq(m^n)$, where both partitions are regarded as having $n$ parts by appending zeros.
Define
\[
\mu^\vee
=
(m-\mu_n,m-\mu_{n-1},\ldots,m-\mu_1).
\]
Thus, $\mu^\vee$ is the Ferrers shape obtained by rotating the
complement of $\mu$ in the rectangle $(m^n)$ by $180^\circ$.

\begin{proposition}
\label{PropSkewFerrersRSK}
Let $\mu\subseteq\lambda\subseteq(m^n)$, and let $A$ be an $n\times m$ matrix with entries in $\mathbb N$.
Suppose that
\[
\operatorname{RSK}(A)=(P,Q),
\]
where $P$ has entries in $\{1,\ldots,m\}$ and $Q$ has entries in
$\{1,\ldots,n\}$.

Then $\operatorname{supp}(A)\subseteq\lambda/\mu$ if and only if $(P,Q)\ \text{is $\lambda$-admissible}$ and $\bigl(\operatorname{ev}_m(P),\operatorname{ev}_n(Q)\bigr)$ is $\mu^\vee$-admissible, where $\operatorname{ev}_r$ denotes the Schützenberger involution on tableaux with entries in $\{1,\ldots,r\}$.
\end{proposition}

\begin{proof}
Let $A^\circ$ denote the matrix obtained from $A$ by rotation through $180^\circ$, so that
\[
(A^\circ)_{ij}=A_{n+1-i,m+1-j}.
\]
We first observe that
\[
\operatorname{supp}(A)\subseteq\lambda/\mu
\quad\Longleftrightarrow\quad
\operatorname{supp}(A)\subseteq\lambda
\ \text{and}\
\operatorname{supp}(A^\circ)\subseteq\mu^\vee.
\]
Indeed, a cell $(i,j)$ belongs to $\lambda/\mu$ precisely when $\mu_i<j\leq\lambda_i$.
The inequality $j\leq\lambda_i$ is exactly the condition $(i,j)\in\lambda$. On the other hand, under the rotation
\[
(i,j)\longmapsto(i',j')
=(n+1-i,m+1-j),
\]
the inequality $j>\mu_i$ becomes
\[
j'=m+1-j\leq m-\mu_i
   =\mu^\vee_{\,n+1-i}
   =\mu^\vee_{\,i'}.
\]
Thus $(i',j')\in\mu^\vee$, proving the claimed equivalence.

Now suppose that $\operatorname{RSK}(A)=(P,Q)$.
By the standard compatibility of RSK with Schützenberger evacuation under $180^\circ$ rotation of the matrix (see, e.g., \cite[Appendix A.1]{Fulton1996}),
\[
\operatorname{RSK}(A^\circ)
=
\bigl(\operatorname{ev}_m(P),
      \operatorname{ev}_n(Q)\bigr).
\]
Hence, by Theorem~\ref{TeoremaKeysLambda},
\[
\operatorname{supp}(A)\subseteq\lambda
\quad\Longleftrightarrow\quad
(P,Q)\text{ is $\lambda$-admissible},
\]
whereas, applying the same theorem to $A^\circ$ and $\mu^\vee$,
\[
\operatorname{supp}(A^\circ)\subseteq\mu^\vee
\quad\Longleftrightarrow\quad
\bigl(\operatorname{ev}_m(P),\operatorname{ev}_n(Q)\bigr)
\text{ is $\mu^\vee$-admissible}.
\]
Combining these equivalences gives the result.
\end{proof}

A pair $(P,Q)$ satisfying the two conditions in Proposition~\ref{PropSkewFerrersRSK} will be called $(\lambda/\mu)$-admissible.

\begin{corollary}
\label{CorSkewFerrersKernel}
\[
\prod_{\substack{1\leq i\leq n\\
                  \mu_i<j\leq\lambda_i}}
\frac{1}{1-x_i y_j}
=
\sum_{\substack{(P,Q)\\
                (P,Q)\text{ is }(\lambda/\mu)\text{-admissible}}}
x^{\operatorname{wt}(Q)}
y^{\operatorname{wt}(P)}.
\]
Here the sum ranges over all $(\lambda/\mu)$-admissible pairs of semistandard tableaux of the same shape.
\end{corollary}

\begin{proof}
Expanding each factor as a geometric series gives
\[
\prod_{\substack{1\leq i\leq n\\
                  \mu_i<j\leq\lambda_i}}
\frac{1}{1-x_i y_j}
=
\sum_{\operatorname{supp}(A)\subseteq\lambda/\mu}
\prod_{i=1}^{n}\prod_{j=1}^{m}
(x_i y_j)^{A_{ij}},
\]
where the sum ranges over all $n\times m$ matrices with entries in $\mathbb N$ whose support is contained in $\lambda/\mu$.

If $\operatorname{RSK}(A)=(P,Q)$, then RSK preserves the row and column sums of $A$. Consequently,
\[
\prod_{i,j}(x_i y_j)^{A_{ij}}
=
x^{\operatorname{wt}(Q)}
y^{\operatorname{wt}(P)}.
\]
Moreover, RSK is a bijection between matrices in $\mathbb N^{n\times m}$ and pairs $(P,Q)$ of semistandard tableaux of the same shape, with $P$ having entries in $\{1,\ldots,m\}$ and $Q$ having entries in $\{1,\ldots,n\}$.

By Proposition~\ref{PropSkewFerrersRSK}, the matrices satisfying
\[
\operatorname{supp}(A)\subseteq\lambda/\mu
\]
correspond precisely to the $(\lambda/\mu)$-admissible pairs.
Therefore
\[
\prod_{\substack{1\leq i\leq n\\
                  \mu_i<j\leq\lambda_i}}
\frac{1}{1-x_i y_j}
=
\sum_{\substack{(P,Q)\\
                (P,Q)\text{ is }(\lambda/\mu)\text{-admissible}}}
x^{\operatorname{wt}(Q)}
y^{\operatorname{wt}(P)},
\]
as claimed.
\end{proof}

It is useful to introduce an auxiliary notation, since the skew condition naturally gives a lower key bound coming from $\mu$ and an upper key bound coming from $\lambda$.

For $\alpha\in\mathbb N^n$, whenever the right-hand side is defined, set
\[
\alpha_\mu
:=
\omega_m\!\left((\omega_n\alpha)^{\mu^\vee}\right)
\in\mathbb N^m.
\]
By Section~\ref{SecLeftKeysEvacuation}, this is the index obtained by translating the $\mu^\vee$-dependent upper right-key bound under evacuation into a lower left-key bound.

For weak compositions $\gamma,\delta \in\mathbb N^r$ satisfying $\gamma^+=\delta^+$, let
\[
K_{\gamma,\delta}(y_1,\ldots,y_r)
=
\sum_{\substack{T\in\operatorname{Tab}_{[r]}(\gamma^+)\\
K(\gamma)\leq K_-(T)\leq K_+(T)\leq K(\delta)}}
y^{\operatorname{wt}(T)}.
\]
When $\gamma\leq\delta$, the polynomials \(K_{\gamma,\delta}\) are the standard basis polynomials of Lascoux and Schützenberger \cite{Lascoux1990Schutzenberger}.
Thus $K_{\gamma,\delta}(y)$ is the generating function of the semistandard tableaux whose left and right keys lie in the interval determined by $K(\gamma)$ and $K(\delta)$.

For $\alpha,\beta\in\mathbb N^r$ satisfying $\alpha^+=\beta^+$,
we will also use the generating function associated with a fixed
pair of extremal keys
\[
 \mathcal{A}_{\alpha,\beta}(x_{1},\ldots,x_{r})
=
\sum_{\substack{T\in\operatorname{Tab}_{[r]}(\alpha^+)\\
K_-(T)=K(\alpha)\\
K_+(T)=K(\beta)}}
x^{\operatorname{wt}(T)}.
\]

\begin{definition}
Let $\mu\subseteq \lambda \subseteq(m^n)$ be partitions. We define
\[
\operatorname{Comp}(\mu,\lambda)
=
\left\{
(\alpha,\beta)\in\mathbb N^n\times\mathbb N^n:
\alpha^+=\beta^+,\;
\alpha\leq\beta,\;
\omega_n\alpha\in\operatorname{Comp}(\mu^\vee),\;
\beta\in\operatorname{Comp}(\lambda),\alpha_{\mu} \leq \beta^{\lambda}
\right\},
\]
where $\leq$ denotes the Bruhat order on weak compositions with the same decreasing rearrangement.
\end{definition}

The diagram illustrates how the left and right keys of $Q$
determine, respectively, the lower and upper key bounds on $P$.
\[
\begin{tikzcd}[column sep=large,row sep=large]
K(\alpha) = K_-(Q)
  \arrow[d, mapsto, "\alpha\mapsto\alpha_\mu"']
&
Q
  \arrow[l, "K_-"']
  \arrow[r, "K_+"]
&
K_+(Q)=K(\beta)
  \arrow[d, mapsto, "\beta\mapsto\beta^\lambda"]
\\
K(\alpha_\mu)
  \arrow[r, "\leq"]
&
K_-(P)\leq K_+(P)
  \arrow[r, "\leq"]
&
K(\beta^\lambda).
\end{tikzcd}
\]

With this notation, the skew Cauchy identity takes the following form.
\begin{corollary}
\label{CorSkewFerrersKeys}
With the notation above,
\[
\prod_{\substack{1\leq i\leq n\\
                  \mu_i<j\leq\lambda_i}}
\frac{1}{1-x_i y_j}
=
\sum_{
(\alpha,\beta)\in\operatorname{Comp}(\mu,\lambda)
}
\mathcal{A}_{\alpha,\beta}(x)\,
K_{\alpha_\mu,\beta^\lambda}(y).
\]
\end{corollary}

\begin{proof}
By Corollary~\ref{CorSkewFerrersKernel}, it suffices to group the $(\lambda/\mu)$-admissible pairs $(P,Q)$ according to the left and right keys of $Q$.
Thus fix $\alpha,\beta\in\mathbb N^n$ such that
\[
K_-(Q)=K(\alpha),
\qquad
K_+(Q)=K(\beta).
\]
Because $K_-(Q)$ and $K_+(Q)$ have the same shape, we have $\alpha^+=\beta^+$.
The inequality $K_-(Q)\leq K_+(Q)$, together with the tableau criterion recalled in Section~\ref{seckeys}, then gives $\alpha\leq\beta$.

The $\lambda$-admissibility of $(P,Q)$ implies that $\beta\in\operatorname{Comp}(\lambda)$ and $K_+(P)\leq K(\beta^\lambda)$.

On the other hand,
\[
K_+\bigl(\operatorname{ev}_n(Q)\bigr)
=
\operatorname{ev}_n\bigl(K_-(Q)\bigr)
=
K(\omega_n\alpha).
\]
Since $\bigl(\operatorname{ev}_m(P),\operatorname{ev}_n(Q)\bigr)$ is $\mu^\vee$-admissible, we have $\omega_n\alpha\in\operatorname{Comp}(\mu^\vee)$ and
\[
K_+\bigl(\operatorname{ev}_m(P)\bigr)
\leq
K\bigl((\omega_n\alpha)^{\mu^\vee}\bigr).
\]
Applying $\operatorname{ev}_m$ to the preceding inequality reverses the entrywise order.
Hence
\[
K_-(P)
\geq
K\!\left(
\omega_m\bigl((\omega_n\alpha)^{\mu^\vee}\bigr)
\right)
=
K(\alpha_\mu).
\]
Therefore
\[
K(\alpha_\mu)
\leq K_-(P)
\leq K_+(P)
\leq K(\beta^\lambda).
\]
In particular, by the tableau criterion recalled in Section~\ref{seckeys}, $\alpha_\mu\leq\beta^\lambda$.
Together with
\[
\alpha^+=\beta^+,\qquad
\alpha\leq\beta,\qquad
\omega_n\alpha\in\operatorname{Comp}(\mu^\vee),
\qquad
\beta\in\operatorname{Comp}(\lambda),
\]
this shows that $(\alpha,\beta)\in\operatorname{Comp}(\mu,\lambda)$.

Conversely, fix
$(\alpha,\beta)\in\operatorname{Comp}(\mu,\lambda)$, let $Q$ satisfy
\[
K_-(Q)=K(\alpha),
\qquad
K_+(Q)=K(\beta),
\]
and let $P$ satisfy
\[
K(\alpha_\mu)
\leq K_-(P)
\leq K_+(P)
\leq K(\beta^\lambda).
\]
The upper bound $K_+(P)\leq K(\beta^\lambda)$ is precisely the $\lambda$-admissibility condition associated with $K_+(Q)=K(\beta)$, so $(P,Q)$ is $\lambda$-admissible.

Similarly, applying $\operatorname{ev}_m$ to $K(\alpha_\mu)\leq K_-(P)$ and using $\omega_m\alpha_\mu=(\omega_n\alpha)^{\mu^\vee}$ gives
\[
K_+\bigl(\operatorname{ev}_m(P)\bigr)
\leq
K\bigl((\omega_n\alpha)^{\mu^\vee}\bigr).
\]
Since $K_{+}\bigl(\operatorname{ev}_n(Q)\bigr)
=
K(\omega_n\alpha)$, this is precisely the $\mu^\vee$-admissibility condition for $\bigl(\operatorname{ev}_m(P),\operatorname{ev}_n(Q)\bigr)$.
Hence $(P,Q)$ is $(\lambda/\mu)$-admissible.

Thus, for each $(\alpha,\beta)\in\operatorname{Comp}(\mu,\lambda)$, the
$(\lambda/\mu)$-admissible pairs with
\[
K_-(Q)=K(\alpha),
\qquad
K_+(Q)=K(\beta)
\]
are precisely the Cartesian product of the tableau sets generating $\mathcal{A}_{\alpha,\beta}(x)$ and $K_{\alpha_\mu,\beta^\lambda}(y)$.
Taking generating functions and summing over $\operatorname{Comp}(\mu,\lambda)$ proves the identity.
\end{proof}

Thus the two boundaries of the skew Ferrers diagram have complementary roles: the inner boundary determines the lower left-key bound, while the outer boundary determines the upper right key bound.

The lower transform can also be computed by the parking construction of Proposition~\ref{Parkingalpha}.

\begin{corollary}[Parking construction of the lower transform]
\label{CorParkingLowerTransform}
Let $\mu\subseteq(m^n)$ and let $\beta\in\mathbb N^n$. Set $\overline{\beta}=\omega_n\beta$.
Then $\beta_\mu$ exists if and only if $\overline{\beta}\in\operatorname{Comp}(\mu^\vee)$.
When this condition holds, apply the parking construction of Proposition~\ref{Parkingalpha} to the pair $(\overline{\beta},\mu^\vee)$, obtaining $\overline{\beta}^{\,\mu^\vee}\in\mathbb N^m$.
Then
\[
\beta_\mu
=
\omega_m\!\left(
\overline{\beta}^{\,\mu^\vee}
\right).
\]
Thus the lower transform associated with the inner boundary $\mu$ is computed by reversing the indexing composition, applying the same parking construction used for an upper Ferrers boundary, and finally reversing the output.
\end{corollary}

\begin{proof}
By definition,
\[
\beta_\mu
=
\omega_m\!\left(
(\omega_n\beta)^{\mu^\vee}
\right).
\]
Hence $\beta_\mu$ exists precisely when $\omega_n\beta\in\operatorname{Comp}(\mu^\vee)$.

When this condition holds, Proposition~\ref{Parkingalpha} computes $(\omega_n\beta)^{\mu^\vee}$ by the parking procedure.
Reversing the resulting composition gives $\beta_\mu$, as claimed.
\end{proof}

\begin{example}
Let
\[
\lambda=(m,m-1,\ldots,1),
\qquad
\mu=(m-1,m-2,\ldots,1),
\]
regarded as partitions in $(m^m)$ by appending zeros.
Then $\lambda/\mu$ consists precisely of the cells $(i,j)$ such that $i+j=m+1$, and $\mu^\vee=\lambda$.

For this staircase, the initial parking positions $\lambda[i]=m+1-i$ are distinct, so every part parks at its initial position. Hence $\gamma^\lambda=\omega_m\gamma$ for every $\gamma\in\mathbb N^m$.
Since $\mu^\vee=\lambda$, we obtain
\[
\alpha_\mu
=\omega_m\!\left((\omega_m\alpha)^\lambda\right)
=\omega_m^{3}\alpha=\omega_m\alpha,
\qquad
\beta^\lambda=\omega_m\beta.
\]
Thus the conditions defining $\operatorname{Comp}(\mu,\lambda)$ include $\alpha\leq\beta$ and $\omega_m\alpha\leq\omega_m\beta$.

Since $\omega_m$ reverses Bruhat order, these inequalities force $\alpha=\beta$.
Conversely, $(\alpha,\alpha)\in\operatorname{Comp}(\mu,\lambda)$ for every weak composition $\alpha\in\mathbb N^m$.
Consequently,
\[
\prod_{i=1}^{m}\frac{1}{1-x_i y_{m+1-i}}
=
\sum_{\alpha\in\mathbb N^m}
\mathcal A_{\alpha,\alpha}(x)\,
K_{\omega_m\alpha,\omega_m\alpha}(y).
\]

The inequalities $K_-(T)\leq T\leq K_+(T)$ show that $K_-(T)=K_+(T)=K(\alpha)$ forces $T=K(\alpha)$.
Hence
\[
\mathcal A_{\alpha,\alpha}(x)=x^\alpha,
\qquad
K_{\omega_m\alpha,\omega_m\alpha}(y)=y^{\omega_m\alpha}.
\]
The kernel identity therefore reduces to
\[
\prod_{i=1}^{m}\frac{1}{1-x_i y_{m+1-i}}
=
\sum_{\alpha\in\mathbb N^m}x^\alpha y^{\omega_m\alpha}.
\]
This description can also be checked directly under RSK.
If the biletter $\binom{i}{m+1-i}$ occurs $\alpha_i$ times, the insertion word is $m^{\alpha_1}(m-1)^{\alpha_2}\cdots1^{\alpha_m}$, where exponents denote repetitions.

Inserting the block corresponding to $i$ adds one cell in each of the columns $1,\ldots,\alpha_i$.
Consequently, column $r$ of $Q$ contains precisely $\{i:\alpha_i\geq r\}$, and column $r$ of $P$ contains precisely $\{m+1-i:\alpha_i\geq r\}$.
These are exactly the columns of $K(\alpha)$ and $K(\omega_m\alpha)$, respectively. Thus,
\[
(P,Q)=\bigl(K(\omega_m\alpha),K(\alpha)\bigr).
\]

Hence both tableaux are keys, and
$P=\operatorname{ev}_m(Q)$.
\end{example}

\section{Specializations and comparison to previous results}
\label{SecSpecializations}

\subsection{Half-bubble sort}
\label{SubsecHalfBubbleSort}

We now compare our construction with that of \cite{FeiginKhoroshkinMakedonskyi2026}.

\begin{proposition}[Comparison with half-bubble-sort]
\label{PropHalfBubbleSortComparison}
Let
\[
\lambda=(\lambda_1,\ldots,\lambda_N),
\qquad
M=\lambda_1,
\qquad
\alpha=(\alpha_1,\ldots,\alpha_N)\in\mathbb N^N.
\]
Define
\[
\overline n=(\lambda_N,\lambda_{N-1},\ldots,\lambda_1)
\qquad\text{and}\qquad
\overline d=(\alpha_N,\alpha_{N-1},\ldots,\alpha_1).
\]
Thus, in the notation of \cite{FeiginKhoroshkinMakedonskyi2026}, the entries of $\overline n$ are the weakly increasing column lengths of the transposed and horizontally reflected diagram of $\lambda$.

Then the following statements hold.

\begin{enumerate}
\item
The composition $\overline d$ is $\overline n$-admissible in the sense of \cite[Definition~1.14]{FeiginKhoroshkinMakedonskyi2026} if and only if $\alpha\in\operatorname{Comp}(\lambda)$.

\item
Whenever these equivalent conditions hold, $\mathsf{hbs}_{\overline n}(\overline d)=\alpha^\lambda$.
\end{enumerate}
\end{proposition}

\begin{proof}
If $\alpha=0$, then $\overline d=0^N$, both admissibility conditions
hold, and
\[
\mathsf{hbs}_{\overline n}(\overline d)
=
0^M
=
\alpha^\lambda.
\]
Suppose, therefore, that $\alpha\neq0$, and write
\[
\operatorname{supp}(\alpha)=\{i_1<\cdots<i_c\}.
\]

By definition, $\overline d$ is $\overline n$-admissible if and only if
\[
\#\{j\leq s: \overline{d}_j\neq0\}\leq \overline{n}_s
\qquad
\text{for every }1\leq s\leq N.
\]
Setting $r=N+1-s$, this becomes
\begin{equation}
\label{EqFKMAdmissibility}
\#\{i\geq r:\alpha_i\neq0\}\leq\lambda[r]
\qquad
\text{for every }1\leq r\leq N.
\end{equation}
Taking $r=i_j$ in~\eqref{EqFKMAdmissibility} gives
\[
c-j+1\leq\lambda[i_j],
\qquad 1\leq j\leq c,
\]
which is precisely the criterion in Proposition~\ref{PropExistenceAlphaLambda}.

Conversely, suppose that these inequalities hold.
Given $r$, if the support of $\alpha$ contains an index weakly larger than $r$, let $i_j$ be its smallest such index.
Then
\[
\#\{i\geq r:\alpha_i\neq0\}
=
c-j+1
\leq\lambda[i_j]
\leq\lambda[r].
\]
If no such index exists, the left-hand side is zero.
Hence \eqref{EqFKMAdmissibility} holds, proving the first assertion.

We now compare the two constructions.
Recall from \cite[Corollary~1.16 and Notation~1.17]{FeiginKhoroshkinMakedonskyi2026} that half-bubble-sort is obtained by iterating their one-step serpentine construction.
At stage $s$, the preceding composition is regarded as a composition of length $\overline{n}_s$ by appending trailing zeros.

For a composition $\gamma$ and an integer $h\geq1$, put
\[
I_h(\gamma)=\{p:\gamma_p\geq h\}.
\]
A step with $\overline{d}_s=0$ leaves all the sets $I_h$ unchanged.

At every nonzero stage, admissibility guarantees that the padded preceding composition contains at least one zero, so \cite[Lemma~1.12]{FeiginKhoroshkinMakedonskyi2026} applies.

Consider, therefore, one nonzero step $\gamma\longmapsto\mu$ in which a positive part $d$ is inserted into a composition $\gamma\in\mathbb N^L$ containing at least one zero, where $L= \overline{n}_{s}$.
The recursion in \cite[Lemma~1.12, Equation~(1.13)]{FeiginKhoroshkinMakedonskyi2026} provides increasing positions $p_0<p_1<\cdots<p_k$ and values
\[
0=u_{-1}<u_0<\cdots<u_k=d
\]
such that
\[
\gamma_{p_t}=u_{t-1},
\qquad
\mu_{p_t}=u_t,
\qquad 0\leq t\leq k,
\]
while all other coordinates remain unchanged.

If $h>d$, then every modified entry remains smaller than $h$, and therefore
\[
I_h(\mu)=I_h(\gamma).
\]
Suppose that $h\leq d$, and let $t$ be the smallest index such that $u_t\geq h$. Then $p_t$ is the unique coordinate that enters $I_h$ during this step.
Indeed, for $q<t$ both $u_{q-1}$ and $u_q$ are smaller than $h$, whereas for $q>t$ both are at least $h$.
Moreover, the recursive definition gives
\[
u_t
=
\min\bigl(\{d\}\cup
          \{\gamma_q:p_t<q\leq L\}\bigr).
\]
Since $u_t\geq h$, every position to the right of $p_t$ already
belongs to $I_h(\gamma)$, whereas
\[
\gamma_{p_t}=u_{t-1}<h.
\]
Consequently,
\[
p_t
=
\max\bigl(\{1,\ldots,L\}\setminus I_h(\gamma)\bigr),
\]
and hence
\begin{equation}
\label{EqHBSLevelSet}
I_h(\mu)
=
I_h(\gamma)\cup
\left\{
\max\bigl(\{1,\ldots,L\}\setminus I_h(\gamma)\bigr)
\right\}.
\end{equation}

Iterating~\eqref{EqHBSLevelSet}, and observing that the zero steps do
not affect the level sets, shows that
\[
I_h\bigl(\mathsf{hbs}_{\overline n}(\overline d)\bigr)
\]
is obtained by parking, for every $s$ such that $\overline{d}_s\geq h$, a car
with bound $\overline{n}_s$ in the rightmost unoccupied positive position weakly
to the left of $\overline{n}_s$.

Under the correspondence
\[
s=N+1-i,
\qquad
\overline{d}_s=\alpha_i,
\qquad
\overline{n}_s=\lambda[i],
\]
the multiset of these bounds is precisely
\[
\{\lambda[i]:\alpha_i\geq h\}.
\]
The order-independence established in the proof of Proposition~\ref{PropStrInclusion} shows that the occupied set of the unrestricted parking procedure depends only on this multiset of bounds.
Since the equivalent admissibility conditions ensure that both parking procedures remain in positive positions, the same conclusion applies here.

In the parking construction of Proposition~\ref{Parkingalpha}, all indices satisfying $\alpha_i\geq h$ are processed before those satisfying $\alpha_i<h$. Its occupied positions after these cars have been processed are precisely the positions containing parts of $\alpha^\lambda$ weakly larger than $h$.
Therefore,
\[
I_h\bigl(\mathsf{hbs}_{\overline n}(\overline d)\bigr)
=
I_h(\alpha^\lambda)
\qquad
\text{for every }h\geq1.
\]
A weak composition is determined by all its level sets $I_h$, and hence
\[
\mathsf{hbs}_{\overline n}(\overline d)=\alpha^\lambda.
\]
\end{proof}

\begin{remark}
Proposition~\ref{PropHalfBubbleSortComparison} shows that the admissibility and half-bubble-sort constructions of \cite{FeiginKhoroshkinMakedonskyi2026} coincide exactly with $\operatorname{Comp}(\lambda)$ and the parking construction of $\alpha^\lambda$, after reversing the row indexing.
More precisely, the proof identifies each half-bubble-sort step with a parking move on every relevant level set $I_h$.
Thus, the parking construction gives a direct combinatorial interpretation of the recursive serpentine construction of Feigin, Khoroshkin, and Makedonskyi.

This also identifies the corresponding Cauchy expansions.
Indeed, let $X=(X_1,\ldots,X_M)$ and $Y=(Y_1,\ldots,Y_N)$ denote the variables in the right Cauchy identity of \cite[Corollary~3.20]{FeiginKhoroshkinMakedonskyi2026}, and make the substitution
\[
X_j=y_j,
\qquad
Y_s=x_{N+1-s}.
\]
Since
\[
(p,s)\in\mathbf Y_{\overline n}
\quad\Longleftrightarrow\quad
p\leq\lambda[N+1-s],
\]
their kernel becomes
\[
\prod_{(i,j)\in\lambda}\frac{1}{1-x_i y_j}.
\]
Moreover, their relation between ordinary and opposite Demazure atoms \cite[Equation~(2.17)]{FeiginKhoroshkinMakedonskyi2026} gives
\[
a^{\overline d}(x_N,\ldots,x_1)
=
a_{\alpha}(x_1,\ldots,x_N)
=
\hat K_\alpha(x),
\]
while Proposition~\ref{PropHalfBubbleSortComparison} gives
\[
\kappa_{\mathsf{hbs}_{\overline n}(\overline d)}(y)
=
K_{\alpha^\lambda}(y).
\]
Thus their right Cauchy identity becomes exactly Corollary~\ref{CorKernellambda}.

The distinction lies in the interpretation and proof.
Their expansion is obtained from filtrations of the symmetric algebra of the Ferrers-supported matrix space, whereas Theorem~\ref{TeoremaKeysLambda} gives a tableau-theoretic realization: it explicitly characterizes the pairs arising under ordinary RSK from biwords supported on $\lambda$.

In particular, when $\lambda=(N^N)$, the construction rearranges $\alpha$ in weakly increasing order, whereas for $\lambda=(N,N-1,\ldots,1)$ it reverses the order of its entries.
These specializations agree with \cite[Examples~1.21 and~1.22]{FeiginKhoroshkinMakedonskyi2026}.
\end{remark}

\subsection{Lascoux's arbitrary-Ferrers expansion}
\label{subsec:LascouxComparison}

We now compare Corollary~\ref{CorKernellambda} with Lascoux's divided-difference-operator formula for an arbitrary Ferrers diagram \cite[Theorem~7]{Lascoux2003}.
We follow the description of the permutations associated with the diagram given in \cite[Section~6.1]{AzenhasEmami2015}.
Lascoux's formula is expressed in terms of divided-difference-operator and does not directly describe the corresponding tableaux.
Our aim is to make this tableau structure explicit by relating Lascoux's expansion to equation~\eqref{EqKernellambda}.

Choose \(n\geq\max\{\ell(\lambda),\lambda_1\}\), and regard \(\lambda\) as a partition contained in the square \((n^n)\) by appending zero parts if necessary.
Let
\[
\rho_t=(t,t-1,\ldots,1)
\]
be the largest staircase contained in \(\lambda\).
The cells of \(\lambda/\rho_t\) are separated into a northwest part and a southeast part by a southwest-northeast diagonal passing through a cell of \(\rho_{t+1}\setminus\lambda\).

Label every cell in row \(r\) of the northwest part by \(r-1\).
Reading these labels column by column, from right to left and from top to bottom, gives a reduced word for a permutation denoted by \(\sigma(\lambda,\mathrm{NW})\).
Similarly, label every cell in column \(c\) of the southeast part by \(c-1\).
Reading these labels row by row, from top to bottom and from right to left, gives a reduced word for a permutation denoted by \(\sigma(\lambda,\mathrm{SE})\).

With these conventions, Lascoux's identity reads
\begin{equation}
\label{EqLascouxFerrers}
\prod_{(i,j)\in\lambda}\frac{1}{1-x_i y_j}
=
\sum_{\mu\in\mathbb N^t}
\left(
 \pi_{\sigma(\lambda,\mathrm{NW})}
 \hat K_\mu(x)
\right)
\left(
 \pi_{\sigma(\lambda,\mathrm{SE})}
 K_{\omega_t\mu}(y)
\right),
\end{equation}
where $\omega_t$ is the longest permutation in $\mathfrak{S}_t$.
Throughout this subsection, all compositions are regarded as elements of $\mathbb{N}^n$ by appending trailing zeros whenever necessary, and all Demazure atoms and key polynomials are taken in $n$ variables.

Corollary~\ref{CorKernellambda} gives a different indexing of the
same kernel:
\begin{equation}
\label{EqOurFerrers}
\prod_{(i,j)\in\lambda}\frac{1}{1-x_i y_j}
=
\sum_{\alpha \in \operatorname{Comp}(\lambda)}
\hat K_\alpha(x)K_{\alpha^\lambda}(y).
\end{equation}

To make the relation between the two expansions explicit, define
the integers \(c_{\mu,\alpha}^{\lambda}\) by
\begin{equation}
\label{EqNWAtomExpansion}
\pi_{\sigma(\lambda,\mathrm{NW})}
\hat K_\mu(x)
=
\sum_{\alpha}
c_{\mu,\alpha}^{\lambda}\hat K_\alpha(x).
\end{equation}
For $\alpha\in\operatorname{Comp}(\lambda)$, comparison of the
coefficient of $\hat K_\alpha(x)$ in
\eqref{EqLascouxFerrers} and \eqref{EqOurFerrers} gives
\begin{equation}
\label{EqLascouxCoefficient}
K_{\alpha^\lambda}(y)
=
\sum_{\mu\in\mathbb N^t}
c_{\mu,\alpha}^{\lambda}
\pi_{\sigma(\lambda,\mathrm{SE})}
K_{\omega_t\mu}(y).
\end{equation}

For $\alpha\notin\operatorname{Comp}(\lambda)$, the corresponding coefficient is zero.
Thus, Lascoux's formula organizes the contributions through the northwest and southeast divided-difference-operators associated with \(\lambda\), whereas Corollary~\ref{CorKernellambda} resolves the same kernel directly in the Demazure-atom basis in the \(x\)-variables.
The strictification of \(B_+^\lambda(\alpha)\) identifies the coefficient of each Demazure atom explicitly: it is either the single key polynomial \(K_{\alpha^\lambda}(y)\), or zero when \(\alpha \not\in \operatorname{Comp}(\lambda)\). In this sense, Theorem~\ref{TeoremaKeysLambda}  provides a tableau interpretation of the Demazure-atom coefficients obtained after expanding Lascoux's arbitrary-Ferrers formula.

\subsection{Staircases and truncated staircases}

We next recover the truncated-staircase identity of Azenhas and Emami as a specialization of Corollary~\ref{CorKernellambda}.

\begin{corollary}[Truncated staircases]
\label{CorTruncatedStaircases}
Let $1\leq k,m\leq n$ and $n+1\leq k+m$, and consider the truncated staircase
\[
\lambda=\lambda(n;k,m)
  =\bigl(m^{\,n-m+1},m-1,m-2,\ldots,n-k+1\bigr).
\]
Equivalently,
\[
\lambda_i=\min\{m,n+1-i\},
\qquad 1\leq i\leq k.
\]
Consequently,
\[
(i,j)\in\lambda
\quad\Longleftrightarrow\quad
i+j\leq n+1,
\qquad
1\leq i\leq k,\quad 1\leq j\leq m.
\]

Suppose first that $k\leq m$.
For $\alpha \in\mathbb N^k$, define $\beta(\alpha)=(\beta_1,\ldots,\beta_k)\in\mathbb N^k$ as follows.
Start with the ordered list
\[
L_k=\omega_{k} \alpha=(\alpha_k,\ldots,\alpha_1).
\]
For $i=k,k-1,\ldots,1$, let
\[
\beta_i
=
\max\left\{
\text{the last $\min\{i,n-m+1\}$ entries of $L_i$}
\right\},
\]
and let $L_{i-1}$ be obtained from $L_i$ by deleting one occurrence of $\beta_i$ among its last $\min\{i,n-m+1\}$ entries.
Then
\[
\alpha^\lambda=(0^{\,m-k},\beta(\alpha)).
\]
Therefore, Corollary~\ref{CorKernellambda} specializes to
\[
\prod_{(i,j)\in\lambda}\frac{1}{1-x_i y_j}
=
\sum_{\alpha\in\mathbb N^k}
\hat K_\alpha(x)\,
K_{(0^{\,m-k},\beta(\alpha))}(y).
\]
This is the truncated-staircase expansion of
\cite[Theorem~6]{AzenhasEmami2015}.

Equivalently, $\beta(\alpha)$ can be obtained from the definition of $\alpha^\lambda$.
Namely, form $K(\alpha)$ and, in each column, replace every entry $i$ by
\[
\lambda_i=\min\{m,n+1-i\},
\]
reverse the column, and apply strictification.
The resulting key tableau is $K\bigl(0^{m-k},\beta(\alpha)\bigr)$.

If $m\leq k$, applying the preceding result to the conjugate truncated staircase $\lambda'$ and interchanging $x$ and $y$ gives
\[
\prod_{(i,j)\in\lambda}\frac{1}{1-x_i y_j}
=
\sum_{\alpha \in\mathbb N^m}
K_{\alpha^{\lambda'}}(x)\,
\hat K_\alpha(y),
\]
where $\alpha^{\lambda'}\in\mathbb N^k$ is obtained by the same procedure after interchanging $k$ and $m$.
\end{corollary}

\begin{proof}
For $1\leq i\leq k$ and $1\leq j\leq m$, we have
\[
j\leq\lambda_i
\quad\Longleftrightarrow\quad
j\leq\min\{m,n+1-i\}
\quad\Longleftrightarrow\quad
i+j\leq n+1.
\]
Thus the $\lambda$-admissible biwords are precisely the biwords supported on the truncated staircase considered in \cite{AzenhasEmami2015}.

Let $C_r$ be column $r$ of $K(\alpha)$. Column $r$ of the key tableau
\[
K(0^{\,n-k},\omega_{k} \alpha)
\]
consists of the entries $n+1-i$ with $i\in C_r$.
Since the entries of the insertion tableau lie in $[m]$, bounding a strictly increasing column by this column is equivalent to bounding it by the weak column obtained by replacing each entry $n+1-i$ by
\[
\min\{m,n+1-i\}=\lambda_i.
\]
The latter is precisely column $r$ of $B_+^\lambda(\alpha)$.

By Proposition~\ref{PropStrictificationBound}, the columnwise strictification of $B_+^\lambda(\alpha)$ is the largest key tableau below this bound.
On the other hand, \cite[Proposition~3]{AzenhasEmami2015} identifies this largest key tableau as
\[
K(0^{\,m-k},\beta(\alpha)),
\]
where $\beta(\alpha)$ is given by the recursive construction above.
Hence $\operatorname{str}\bigl(B_+^\lambda(\alpha)\bigr)
=
K(0^{\,m-k},\beta(\alpha))$,
and therefore
\[
\alpha^\lambda=(0^{\,m-k},\beta(\alpha)).
\]
The first identity now follows from Corollary~\ref{CorKernellambda}.
The case $m \leq k$ follows by applying the preceding argument to the conjugate partition $\lambda '$ and interchanging the two alphabets.
\end{proof}

When $k=m=n$, the shape $\lambda$ is the full staircase $(n,n-1,\ldots,1)$ and $\beta(\alpha)=\omega_{k}\alpha$.
When $k\leq m$ and $k+m=n+1$, the shape $\lambda$ is the rectangle $(m^k)$ and $\beta(\alpha)=\omega_{k}\alpha^+$, so the corresponding key polynomial is a Schur polynomial.
The case $m\leq k$ follows symmetrically.

\subsection{$m$-symmetric Cauchy identity}

The $m$-symmetric Schur functions were introduced in~\cite{mSym} as two bases related by duality.
A combinatorial realization of this duality at $t=0$ was developed in~\cite{mSchurt0}, yielding a tableau interpretation and a bijective proof of the corresponding Cauchy identity.
The present work arose from investigating whether the machinery developed there extends to Cauchy kernels supported on other Ferrers shapes.
We now show how the identity of~\cite{mSchurt0} follows from our general construction.

Recall that the ring $R_m$ of $m$-symmetric functions consists of functions that are symmetric in the variables $x_{m+1},x_{m+2},\ldots$, with no symmetry imposed on $x_1,\ldots,x_m$.
Its $m$-symmetric Schur functions are indexed by $m$-partitions $\Lambda=( \pmb a;\nu)$, where $\pmb a\in\mathbb N^m$ and $\nu$ is a partition.
We will only use their specialization at $t=0$ and the corresponding dual basis; see \cite{mSchurt0}.

To match the standard \(m\)-symmetric notation, throughout this subsection \(m\) denotes the number of distinguished variables; accordingly, the Ferrers diagram below has width \(N+m\).

\begin{corollary}
\label{RemMSchurSpecialization}
Let $N\geq m$, with $x=(x_1,\ldots,x_N)$ and $y=(y_1,\ldots,y_N)$.
Then
\[
\left(\prod_{i,j=1}^{N}\frac{1}{1-x_i y_j}\right)
\left(
\prod_{\substack{1\leq i,h\leq m\\i+h\leq m+1}}
\frac{1}{1-x_i\hat y_h}
\right)
=
\sum_{\substack{\Lambda=(\mathbf a;\nu)\\
                \ell(\nu)\leq N-m}}
s_\Lambda(x;0)\,
\mathfrak{s}^{*}_\Lambda(y,\hat y).
\]
\end{corollary}

\begin{proof}
Equation~\eqref{EqKernellambda} specializes to the Cauchy identity proved in \cite{mSchurt0}.
To see this, fix $m\leq N$ and consider the partition
\[
\lambda
=
\bigl(N+m,N+m-1,\ldots,N+1,N^{N-m}\bigr);
\]
equivalently,
\[
\lambda_i=
\begin{cases}
N+m+1-i, & 1\leq i\leq m,\\
N,       & m+1\leq i\leq N.
\end{cases}
\]
Identify the letters
\[
N+1<N+2<\cdots<N+m
\]
with
\[
\hat{1}<\hat{2}<\cdots<\hat{m},
\]
and accordingly set $y_{N+j}=\hat y_j$ for $1\leq j\leq m$.
Since $\lambda_i\geq N$ for every $i$, we have $\operatorname{Comp}(\lambda)=\mathbb N^N$.
Under this identification, the kernel associated with $\lambda$ becomes
\[
\prod_{(i,j)\in \lambda}\frac{1}{1-x_i y_j}
=
\left(
\prod_{\substack{1\leq i,j\leq m\\ i+j\leq m+1}}
\frac{1}{1-x_i\hat y_j}
\right)
\left(
\prod_{i,j=1}^{N}\frac{1}{1-x_i y_j}
\right),
\]
which is precisely the kernel appearing in \cite{mSchurt0}.

Let $
I=\{i:\lambda_i=N\}=\{m+1,\ldots,N\}$.
Since the boundary function $i\mapsto\lambda_i$ is constant on $I$,
permuting the entries of a composition in the positions belonging to $I$ does not change the corresponding $\lambda$-dependent key.
Therefore, for every $\alpha\in\mathbb N^N$ and every $\sigma\in\mathfrak S_I$, we have $\alpha^\lambda=(\sigma\alpha)^\lambda $.

More explicitly, write $\alpha=(a_1,\ldots,a_m,b_1,\ldots,b_{N-m})$, and let $\nu=(\nu_1,\ldots,\nu_\ell)$ be the partition obtained by rearranging $(b_1,\ldots,b_{N-m})$ in weakly decreasing order and deleting its zero parts.
Define the $m$-partition
\[
\Lambda=(\mathbf a;\nu),
\qquad
\mathbf a=(a_1,\ldots,a_m).
\]
The columnwise $\lambda$-operation followed by strictification gives
\[
\alpha^\lambda
=
\bigl(
0^{N-\ell},
\nu_\ell,\ldots,\nu_1,
a_m,\ldots,a_1
\bigr).
\]
In particular, $\alpha^\lambda$ depends on $(b_1,\ldots,b_{N-m})$ only through the partition $\nu$.

Let $\mathcal O_I(\alpha)$ denote the set of distinct compositions obtained from $\alpha$ by permuting the entries in the positions $m+1,\ldots,N$.
The tableau descriptions in \cite{mSchurt0} then give
\[
\sum_{\beta\in\mathcal O_I(\alpha)}
\hat K_\beta(x)
=
s_{\Lambda}(x;0)
\]
and
\[
K_{\alpha^\lambda}
 \bigl(y_1,\ldots,y_N,\hat y_1,\ldots,\hat y_m\bigr)
=
\mathfrak{s}^{*}_{\Lambda}(y,\hat y).
\]
Consequently, grouping the terms in \eqref{EqKernellambda} into $\mathfrak S_I$-orbits yields
\[
\begin{aligned}
\prod_{(i,j)\in\lambda}\frac{1}{1-x_i y_j}
&=
\sum_{\alpha\in\mathbb N^N}
\hat K_\alpha(x)K_{\alpha^\lambda}(y)\\
&=
\sum_{\substack{\Lambda=(\mathbf a;\nu)\\
                 \ell(\nu)\leq N-m}}
s_{\Lambda}(x;0)
\mathfrak{s}^{*}_{\Lambda}(y,\hat y),
\end{aligned}
\]
which recovers the identity proved in \cite{mSchurt0}.
\end{proof}

\appendix

\section{Infinite alphabets}
\label{AppInfiniteAlphabet}

The admissibility condition in the main text is an inequality between a right key and a weak upper bound.
In a finite alphabet, strictification replaces that bound by a largest key tableau whenever the bounded family is nonempty.
Over an infinite alphabet, the bounded family may remain nonempty even when no largest key tableau exists.
Nevertheless, the weak-bound formulation remains valid and yields an infinite-variable Cauchy identity for the $m$-symmetric Schur functions.

Fix $m\geq1$ and consider
\[
\hat{\mathcal A}
=\mathbb Z_{>0}\cup\{\hat1,\ldots,\hat m\},
\qquad
\hat{\mathcal A}_0=\hat{\mathcal A}\cup\{\hat0\},
\]
with the order
\[
1<2<\cdots<\hat0<\hat1<\cdots<\hat m.
\]
The letter $\hat0$ is used only as an upper bound and is never a
tableau entry.
All tableaux have finitely many cells; therefore their reading words,
and the sets of subwords used to define their right keys, are finite.

Define the weakly decreasing upper-bound function
\[
\lambda:\mathbb Z_{>0}\longrightarrow\hat{\mathcal A}_0,
\qquad
\lambda(i)=
\begin{cases}
\widehat{m+1-i},&1\leq i\leq m,\\
\hat0,&i>m.
\end{cases}
\]
A biletter $\binom{i}{j}$, with $i\in\mathbb Z_{>0}$ and
$j\in\hat{\mathcal A}$, is admissible when $j\leq\lambda(i)$.
Every ordinary bottom letter is allowed.
A hatted bottom letter $j=\hat h$ is allowed precisely when
$i+h\leq m+1$.

For a weak composition $\alpha$ of finite support, define
$B_+^\lambda(\alpha)$ by applying $\lambda$ to the entries of
$K(\alpha)$ and reversing each column, as in the finite setting.
This gives a weak upper-bound filling of shape $\alpha^+$.
We will use this filling directly, without requiring a strictification.

\smallskip
\noindent
\textit{Bounds indexed by $m$-partitions.}
Let $\Lambda=(\mathbf a;\nu)$ be an $m$-partition, where
$\mathbf a=(a_1,\ldots,a_m)$, and let $\Lambda^{(0)}$ be the
partition obtained by sorting the parts of $\mathbf a$ and $\nu$.
Define $B_+(\Lambda)$ to be the weak filling of $\Lambda^{(0)}$
whose column $r$ contains the letters
\[
\{i\in\{1,\ldots,m\}:a_i\geq r\},
\]
listed increasingly, followed by enough copies of $\hat0$ to fill that column.
Define $B_+^\lambda(\Lambda)$ by applying $\lambda$ to the ordinary
entries, leaving $\hat0$ unchanged, and reversing each column.

For $\alpha=(a_1,\ldots,a_m,b_1,b_2,\ldots)$, let $\nu(\alpha)$ be the partition obtained by sorting the nonzero parts of $(b_1,b_2,\ldots)$, and set
\[
\Lambda(\alpha)=((a_1,\ldots,a_m);\nu(\alpha)).
\]
The indices in $\{1,\ldots,m\}$ occurring in each column of
$K(\alpha)$ are determined by $\mathbf a$.
All remaining indices are sent to $\hat0$, and their number in
column $r$ is determined by $\nu(\alpha)$.
Consequently,
\[
\alpha^+=\Lambda(\alpha)^{(0)},\qquad
B_+^\lambda(\alpha)=B_+^\lambda\bigl(\Lambda(\alpha)\bigr).
\]
Thus the transformed bound depends only on $\Lambda(\alpha)$.

\begin{example}
Let $m=3$ and $\Lambda=((0,2,1);(3))$, so that
$\Lambda^{(0)}=(3,2,1)$.
Then
\[
B_+(\Lambda)=
\tableau[scY]{2&2&\hat0\\3&\hat0\\\hat0},
\qquad
B_+^\lambda(\Lambda)=
\tableau[scY]{\hat0&\hat0&\hat0\\\hat1&\hat2\\\hat2}.
\]
For every ordinary positive integer $R$, the tableau
\[
T_R=\tableau[scY]{R&R&R\\\hat1&\hat2\\\hat2}
\]
is a key tableau and satisfies $T_R\leq B_+^\lambda(\Lambda)$.
The upper-left entry of any tableau below this bound must be an
ordinary integer, whereas the upper-left entries of the $T_R$ are
unbounded. Hence there is no largest key tableau below the bound,
although the bounded family is nonempty.
\end{example}

\begin{definition}
For an $m$-partition $\Lambda$, define
\[
\mathcal S(\Lambda)=
\left\{
Q\in\operatorname{Tab}_{\mathbb Z_{>0}}(\Lambda^{(0)}):
K_+(Q)\big|_{\{1,\ldots,m\}}
=B_+(\Lambda)\big|_{\{1,\ldots,m\}}
\right\},
\]
and
\[
\mathcal S^*(\Lambda)=
\left\{
P\in\operatorname{Tab}_{\hat{\mathcal A}}(\Lambda^{(0)}):
K_+(P)\leq B_+^\lambda(\Lambda)
\right\}.
\]
Here the restriction of a filling to $\{1,\ldots,m\}$ retains the
positions and values of the entries in that set and ignores all others.
\end{definition}

We extend to this infinite setting the tableau descriptions of the $m$-symmetric Schur functions at $t=0$ and their duals in separated alphabets \cite[Equations~(5.5), (5.8), and Remark~33]{mSchurt0}:
\[
s_\Lambda(x;0)=\sum_{Q\in\mathcal S(\Lambda)}x^{\operatorname{wt}(Q)},
\qquad
\mathfrak s^*_\Lambda(y,\hat y)
=\sum_{P\in\mathcal S^*(\Lambda)}
y^{\operatorname{wt}_{\mathbb Z_{>0}}(P)}
\hat y^{\operatorname{wt}_{\{\hat1,\ldots,\hat m\}}(P)}.
\]

\smallskip
\noindent
\textbf{The Cauchy identity.}
For a weak composition $\alpha$ of finite support, write
\[
F_\alpha(y,\hat y)
=\sum_{\substack{
P\in\operatorname{Tab}_{\hat{\mathcal A}}(\alpha^+)\\
K_+(P)\leq B_+^\lambda(\alpha)}}
y^{\operatorname{wt}_{\mathbb Z_{>0}}(P)}
\hat y^{\operatorname{wt}_{\{\hat1,\ldots,\hat m\}}(P)}.
\]
The equality of bounds above gives $F_\alpha(y,\hat y)=\mathfrak s^*_{\Lambda(\alpha)}(y,\hat y)$.

We also have the disjoint decomposition
\[
\mathcal S(\Lambda)
=\bigsqcup_{\Lambda(\alpha)=\Lambda}
\left\{Q\in\operatorname{Tab}_{\mathbb Z_{>0}}(\alpha^+):
K_+(Q)=K(\alpha)\right\}.
\]
Indeed, the entries of $K_+(Q)$ at most $m$ determine the first
$m$ parts of its indexing composition, and the shape determines
the multiset of the remaining parts.
The tableau formula for Demazure atoms therefore gives
\[
s_\Lambda(x;0)
=
\sum_{\Lambda(\alpha)=\Lambda}\hat K_\alpha(x),
\]
where, for each \(\alpha\), \(\hat K_\alpha\) is computed in the variables \(x_1,\ldots,x_N\) for any \(N\) large enough that \(\operatorname{supp}(\alpha)\subseteq[N]\); appending trailing zero parts does not change it.
Compositions in these sums are viewed as eventually zero sequences, so that different numbers of trailing zeros do not create additional indices.

The restricted-RSK correspondence extends to the present alphabets
by finite truncation.
To see this explicitly, fix $N\geq m$ and $L\geq1$, restrict the
top letters to $\{1,\ldots,N\}$ and the bottom letters to
$\{1,\ldots,L,\hat1,\ldots,\hat m\}$, and relabel
$\hat h$ as $L+h$.
The admissibility condition is then the finite Ferrers condition for
\[
\lambda^{(N,L)}=(L+m,L+m-1,\ldots,L+1,L^{N-m}).
\]
Replacing a bound $\hat0$ by $L$ makes the right-key inequalities
exactly the finite inequalities for this partition.
Thus Theorem~\ref{TeoremaKeysLambda} applies.
Every biword and every tableau pair under consideration occurs in
some such truncation, so the finite bijections give a bijection on
the full alphabets.

Writing $y_{\hat h}=\hat y_h$ and taking generating series yields
\begin{proposition}
\begin{equation}
\label{EqInfiniteKernel}
\begin{aligned}
\prod_{\substack{i\in\mathbb Z_{>0},\ j\in\hat{\mathcal A}\\
                  j\leq\lambda(i)}}\frac{1}{1-x_i y_j}
&=\left(\prod_{i,j\geq1}\frac{1}{1-x_i y_j}\right)
  \left(\prod_{\substack{1\leq i,h\leq m\\i+h\leq m+1}}
        \frac{1}{1-x_i\hat y_h}\right)\\
&=\sum_\alpha\hat K_\alpha(x)F_\alpha(y,\hat y)\\
&=\sum_\Lambda s_\Lambda(x;0)\mathfrak s^*_\Lambda(y,\hat y).
\end{aligned}
\end{equation}
The first sum runs over weak compositions of finite support and
the second over $m$-partitions.
\end{proposition}
The last equality follows by grouping the first sum according to
$\Lambda(\alpha)$.

\enlargethispage{2\baselineskip}
\bibliographystyle{plain}
\bibliography{Biblio}

@article{mSym,
  author  = {Lapointe, Luc},
  title   = {$m$-symmetric functions, non-symmetric {M}acdonald polynomials and positivity conjectures},
  journal = {Trans. Amer. Math. Soc.},
  volume  = {378},
  number  = {12},
  pages   = {8319--8359},
  year    = {2025}
}

@misc{mSchurt0,
  author = {Lapointe, Luc and Pena, Luis},
  title  = {A new characterization of right keys, and the
            $m$-symmetric {S}chur functions at $t=0$},
  year   = {2026},
  note   = {arXiv:2608.13276}
}

@article{Demazure1974a,
  author  = {Demazure, Michel},
  title   = {D\'esingularisation des vari\'et\'es de {S}chubert g\'en\'eralis\'ees},
  journal = {Ann. Sci. \'Ecole Norm. Sup. (4)},
  volume  = {7},
  pages   = {53--88},
  year    = {1974}
}

@article{Demazure1974b,
  author  = {Demazure, Michel},
  title   = {Une nouvelle formule des caract\`eres},
  journal = {Bull. Sci. Math. (2)},
  volume  = {98},
  number  = {3},
  pages   = {163--172},
  year    = {1974}
}

@incollection{Lascoux1990Schutzenberger,
  author    = {Lascoux, Alain and Sch\"utzenberger, Marcel-Paul},
  title     = {Keys \& standard bases},
  booktitle = {Invariant Theory and Tableaux (Minneapolis, MN, 1988)},
  series    = {IMA Vol. Math. Appl.},
  volume    = {19},
  pages     = {125--144},
  publisher = {Springer, New York},
  year      = {1990}
}

@article{Willis2013,
  author  = {Willis, Matthew J.},
  title   = {A direct way to find the right key of a semistandard {Y}oung tableau},
  journal = {Ann. Comb.},
  volume  = {17},
  number  = {2},
  pages   = {393--400},
  year    = {2013}
}

@book{Fulton1996,
  author    = {Fulton, William},
  title     = {Young Tableaux. {W}ith Applications to Representation Theory and Geometry},
  series    = {London Mathematical Society Student Texts},
  volume    = {35},
  publisher = {Cambridge University Press, Cambridge},
  year      = {1997}
}

@book{Stanley_Fomin_1999,
  author    = {Stanley, Richard P.},
  title     = {Enumerative Combinatorics. {V}ol. 2},
  series    = {Cambridge Studies in Advanced Mathematics},
  volume    = {62},
  note      = {With a foreword by Gian-Carlo Rota and appendix 1 by Sergey Fomin},
  publisher = {Cambridge University Press, Cambridge},
  year      = {1999}
}

@incollection{Lascoux2003,
  author    = {Lascoux, Alain},
  title     = {Double {C}rystal {G}raphs},
  booktitle = {Studies in Memory of {I}ssai {S}chur},
  series    = {Progr. Math.},
  volume    = {210},
  pages     = {95--114},
  publisher = {Birkh\"auser Boston, Boston, MA},
  year      = {2003}
}

@article{FuLascoux2009,
  author  = {Fu, Amy M. and Lascoux, Alain},
  title   = {Non-symmetric {C}auchy kernels for the classical groups},
  journal = {J. Combin. Theory Ser. A},
  volume  = {116},
  number  = {4},
  pages   = {903--917},
  year    = {2009}
}

@article{AzenhasEmami2015,
  author  = {Azenhas, Olga and Emami, Aram},
  title   = {An analogue of the {R}obinson-{S}chensted-{K}nuth correspondence and non-symmetric {C}auchy kernels for truncated staircases},
  journal = {European J. Combin.},
  volume  = {46},
  pages   = {16--44},
  year    = {2015}
}

@article{Mason2008,
  author  = {Mason, Sarah},
  title   = {A Decomposition of {S}chur Functions and an Analogue of the
             {Robinson--Schensted--Knuth} Algorithm},
  journal = {S{\'e}minaire Lotharingien de Combinatoire},
  volume  = {57},
  pages   = {Article B57e, 24 pp.},
  year    = {2008}
}

@incollection{AzenhasEmami2015Growth,
  author    = {Azenhas, Olga and Emami, Aram},
  title     = {Growth Diagrams and Non-symmetric {C}auchy Identities on {NW} ({SE}) Near Staircases},
  booktitle = {Dynamics, Games and Science},
  series    = {CIM Series in Mathematical Sciences},
  pages     = {41--69},
  publisher = {Springer},
  address   = {Cham},
  year      = {2015},
  doi       = {10.1007/978-3-319-16118-1_4}
}

@misc{AzenhasEmami2014NWSE,
  author        = {Azenhas, Olga and Emami, Aram},
  title         = {{NW--SE} Expansions of Non-symmetric {C}auchy Kernels on Near Staircases and Growth Diagrams},
  year          = {2014},
  eprint        = {1412.0420},
  note    = {arXiv:1412.0420}
}

@article{FeiginKhoroshkinMakedonskyi2026,
  author  = {Feigin, Evgeny and Khoroshkin, Anton and Makedonskyi, Ievgen},
  title   = {Cauchy Identities for Staircase Matrices},
  journal = {Journal of the London Mathematical Society},
  volume  = {114},
  number  = {3},
  pages   = {e70680},
  year    = {2026},
  doi     = {10.1112/jlms.70680}
}

@misc{KhoroshkinMakedonskyi2025,
  author  = {Anton Khoroshkin and Ievgen Makedonskyi},
  title   = {Bubble sort and {Howe} duality for staircase matrices},
  note = {arXiv:2502.21184},
  year    = {2025}
}

@article{ChoiKwon2018,
  author  = {Seung-Il Choi and Jae-Hoon Kwon},
  title   = {{L}akshmibai--{S}eshadri paths and non-symmetric {C}auchy identity},
  journal = {Algebras and Representation Theory},
  volume  = {21},
  number  = {6},
  pages   = {1381--1394},
  year    = {2018},
  doi     = {10.1007/s10468-017-9752-6}
}

@article{AzenhasGobetLecouvey2024,
  author  = {Olga Azenhas and Thomas Gobet and C{\'e}dric Lecouvey},
  title   = {Non symmetric {C}auchy kernel, crystals and last passage
             percolation},
  journal = {Tunisian Journal of Mathematics},
  volume  = {6},
  number  = {2},
  pages   = {249--297},
  year    = {2024},
  doi     = {10.2140/tunis.2024.6.249}
}

@misc{AzenhasGonzalezHuangTorres2024,
  author        = {Azenhas, Olga and Gonz{\'a}lez, Nicolle and
                   Huang, Daoji and Torres, Jacinta},
  title         = {Keys and Evacuation via Virtualization},
  year          = {2024},
  eprint        = {2409.12666},
  archivePrefix = {arXiv},
  primaryClass  = {math.CO},
  note = {arXiv:2409.12666v2}
}

@book{BjornerBrenti2005,
  author    = {Bj{\"o}rner, Anders and Brenti, Francesco},
  title     = {{C}ombinatorics of {C}oxeter {G}roups},
  series    = {Graduate Texts in Mathematics},
  volume    = {231},
  publisher = {Springer},
  address   = {New York},
  year      = {2005},
  doi       = {10.1007/3-540-27596-7}
}

@article{KushwahaRaghavanViswanath2025,
  author  = {Kushwaha, Mrigendra Singh and Raghavan, K. N. and Viswanath, Sankaran},
  title   = {Simple procedures for left and right keys of
             semi-standard {Y}oung tableaux},
  journal = {Algebras and Representation Theory},
  volume  = {28},
  number  = {6},
  pages   = {1407--1429},
  year    = {2025},
  doi     = {10.1007/s10468-024-10299-1}
}

@article{AssafQuijada2018,
  author  = {Assaf, Sami and Quijada, Danjoseph},
  title   = {A {P}ieri rule for key polynomials},
  journal = {S{\'e}minaire Lotharingien de Combinatoire},
  volume  = {80B},
  year    = {2018},
  pages   = {Article 78, 12 pp.}
}

\end{document}